\documentclass[11pt]{amsart}
\usepackage[left=2cm, right=2cm, top=2cm]{geometry}
\usepackage{graphicx}
\usepackage{amssymb}
\usepackage{latexsym}
\usepackage{color}
\usepackage{cite} 
\usepackage{verbatim}
\usepackage{url}
\newtheorem{theorem}{Theorem}[section]

\newtheorem{prop}{Proposition}[section]
\newtheorem{corollary}[theorem]{Corollary}
\usepackage{comment}
\theoremstyle{definition}

\newtheorem{example}[theorem]{Example}

\theoremstyle{remark}

\numberwithin{equation}{section}

\begin{document}

\title{Steklov Eigenvalue Problems on Nearly Spherical and Nearly Annular Domains}\thanks{The authors acknowledge partial support from NSF grant DMS-2208373.}

\author{Nathan Schroeder}
\address{E-mail: nathan.schroeder@cgu.edu; Address: Institute of Mathematical
Sciences, Claremont Graduate University, Claremont, CA 91711}

\author{Weaam Alhejaili}
\address{E-mail: weaalhejali@pnu.edu.sa; Address: Department of Mathematical Sciences, College of Science, Princess Nourah bint Abdulrahman University, P.O.Box 84428, Riyadh 11671, Saudi Arabia}

\author{Chiu-Yen Kao}
\address{E-mail: ckao@cmc.edu; Address: Department of Mathematical Sciences,
Claremont McKenna College, Claremont, CA 91711}

\subjclass{35P15,49Q10,65N25}
\keywords{Steklov eigenvalue; eigenvalue optimization; spherical harmonic functions; addition theorem}

%

\begin{abstract}
We consider Steklov eigenvalues on nearly spherical and nearly annular domains in $d$ dimensions. By using the Green-Beltrami identity for spherical harmonic functions, the derivatives of Steklov eigenvalues with respect to the domain perturbation parameter can be determined by the eigenvalues of a matrix involving the integral of the product of three spherical harmonic functions. By using the addition theorem for spherical harmonic functions, we determine conditions when the trace of this matrix becomes zero. These conditions can then be used to determine when spherical and annular regions are critical points while we optimize Steklov eigenvalues subject to a volume constraint. In addition, we develop numerical approaches based on particular solutions and show that numerical results in two and three dimensions are in agreement with our analytic results. 
\end{abstract}
%
%

%
\maketitle
\section{Introduction}

We consider the Steklov eigenvalue problem on  a bounded open set $\Omega \subset \mathbb{R}^d$:
\begin{equation}
\begin{cases}\label{E:StekProbGen}
\Delta u \ = \ 0 & \text{in}\ \Omega,\\
\partial_n u \ = \ \sigma u & \text{on}\ \partial \Omega.
\end{cases}
\end{equation}\\
\noindent This problem appears while modeling the dynamics of a sloshing liquid  \cite{kuznetsov2014legacy} and while studying the Dirichlet-to-Neumann map \cite{uhlmann2014inverse}. A number $\sigma$ for which the Steklov problem has a non-trivial solution $u$ is called a \textit{Steklov eigenvalue of $\Omega$} and the collection of all Steklov eigenvalues is called the \textit{Steklov spectrum of $\Omega$}.  If we further assume that $\partial \Omega$ is Lipschitz,  then the Steklov spectrum is discrete starting with the eigenvalue 0 and proceeding through an increasing sequence of  finite multiplicity eigenvalues that diverge to infinity, i.e., $ 0 =  \sigma_0 \ \le \ \sigma_1 \ \le \ \sigma_2 \ \le \ \cdots \ \nearrow \infty$ \cite{bandle1980isoperimetric}. The eigenvalues have a variational characterization, 
\begin{equation} \label{eq: lambda_k}
\sigma_{k}(\Omega)=
\min_{v\in H^{1}(\Omega)} \ 
\left\{ \frac{\int_{\Omega}\left|\nabla v \right|^{2}dx}{\int_{\partial\Omega}v^{2}ds}\colon \  
\int_{\partial\Omega}vu_{j}=0, \ 
j=0,\ldots,k-1\right\}
\end{equation} 
where $u_j$ is the corresponding $j-$th eigenfunction.  It is well-known that, by using separation of variables \cite{girouard2017spectral, dittmar2004sums, martel2014spectre}, the Steklov spectrum can be determined explicitly on balls and annular domains (See Example \ref{E:dimdBallEigs}, Example \ref{E:ExampleAnnEigs}, and Appendix \ref{AnnEigFormulas}).

The problem of finding a global optimizing shape for a Steklov eigenvalue $\sigma_k$ among a constrained family of admissible shapes is classic and has been studied extensively.  For instance, Weinstock \cite{weinstock1954inequalities} showed that, among planar domains with fixed perimeter, the first Steklov eigenvalue $\sigma_1$ is maximized by a disk provided that $\Omega$ is simply-connected. Furthermore, for simply-connected planar domains with fixed perimeter, the $k$-th Steklov eigenvalue with $k\ge 1$ is maximized in the limit by a disjoint union of $k$ identical disks \cite{girouard2010}. In higher dimension $d\ge3$, Brock  \cite{brock2001isoperimetric} showed that the ball maximizes $\sigma_1$ among open sets of a given volume.   Ftouhi \cite{ftouhi2022place}  considers the family of  doubly connected domains of the form $\mathbb{B}_1\setminus \mathbb{B}_2$ with $\overline{\mathbb{B}_2}\subset \mathbb{B}_1$,  where $\mathbb{B}_1$ and  $\mathbb{B}_2$ are not necessarily concentric $d$-dimensional balls.  It is shown that among all such domains  the first Steklov eigenvalue $\sigma_1$ is maximized uniquely when the balls are concentric. A comprehensive review of these types of Steklov eigenvalue problems can be found in  \cite{girouard2017spectral}. 

An alternative type of  Steklov eigenvalue problem asks when a given initial shape $\Omega$ locally optimizes a Steklov eigenvalue $\sigma_k$ among all nearby shapes that are small perturbations of $\Omega$ by members of some fixed class of deformation fields.   In this paper, we study the local optimization problem in $d\geq2$ dimensions where the initial domain is either spherical or annular; and  the nearby perturbed \textit{nearly spherical} or \textit{nearly annular} domains are induced by a  smooth volume preserving at first order deformation field.  We base our approach on fundamental results from Dambrine et al. \cite{dambrine2014extremal}, where they show, for instance,  that any ball in dimension $d=2$ or $d=3$ locally maximizes the first Steklov eigenvalue $\sigma_1$ under smooth, volume preserving perturbation.   Viator and Osting \cite[Theorem 1.1]{viator2022steklov} show in dimension $d=3$ that  for any $k = 1,2,3,\cdots$,  a ball $B$ in $\mathbb{R}^3$ is stationary for $\sigma_{k^2}$.  In dimension $d=2$, Quinones shows that an annulus is, in an appropriate sense, locally critical for the $\sigma_1$ when perturbed by smooth, perimeter length preserving deformation fields, see \cite[Proposition 5]{quinones2019critical}. Viator and Osting rule out any disk in dimension $d=2$ as a  local maximizer of  the Steklov eigenvalues $\sigma_{2k}$, for $k =1,2,3,\cdots$, see the discussion following \cite[Theorem 4,3]{viator2018steklov}.  

Turning to numerical techniques, we note that  to solve extremal Steklov eigenvalue problems most approaches start with an initial guess of the domain and deform it iteratively based on a gradient-ascent approach until it converges to an optimal domain. Numerical methods based on finite element approaches \cite{bonder2007optimization,sayed2021maximization} can handle complex geometries and allow adaptive meshes to improve accuracy and efficiency but they require a mesh on the whole domain $\Omega$.  To reduce the computational cost and achieve high accuracy, methods which only require discretization of the boundary $\partial \Omega$ are preferred, e.g., conformal mapping approaches \cite{alhejaili2019numerical,alhejaili2019maximal,oudet2021computation,kao2023computational}, boundary integral methods \cite{akhmetgaliyev2017computational,ammari2020optimization}, method of particular solutions (MPS)\cite{oudet2021computation,kao2023computational}, and method of fundamental solutions (MFS) \cite{bogosel2016method,bogosel2017optimal,antunes2021numerical}. Advances in these numerical techniques enabled the discovery of local maximizers of $\sigma_k$ subject to a fixed volume constraint in dimension two 
\cite{bogosel2016method,akhmetgaliyev2017computational,alhejaili2019numerical} and higher \cite{antunes2021numerical}. In two dimensions, the optimal domains of $\sigma_k$ looks like a ruffled edge pie dish. Comparable results are observed in both three- and four-dimensional calculations by MFS \cite{antunes2021numerical}. Recently, MPS methods have been developed to solve maximal Steklov eigenvalues problem among two-dimensional surfaces with zero genus and several boundary components \cite{oudet2021computation,kao2023computational}. In this paper, we further develop MPS methods to compute Steklov eigenvalues in both two and three dimensions and use it to study local perturbation problems. 

The organization and contributions of this paper are as follows. In Section 2, we review Steklov eigenvalues on a ball and an annular domain, referring to  Appendix \ref{AnnEigFormulas}, where we present general formulas for the eigenvalues and eigenvectors of annular domains in $d-$dimensions.  The remainder of section 2 reviews properties of spherical harmonic functions, ending with a derivation of  a triple product integral identity for $d$-dimensional spherical harmonic functions. In Section 3, we discuss  perturbations of Steklov eigenvalues. In particular, we review fundamental results from Dambrine-Kateb-Lamboley \cite{dambrine2014extremal} and restate their theorem concerning the indices of Steklov eigenvalues being reordered so that the eigenvalues become differentiable. We then introduce the subdifferential of eigenvalues and discuss critical conditions. In Section 4, the perturbation of Steklov eigenvalues on nearly spherical and nearly annular domains in $\mathbb{R}^d$ are derived. In Section 5 we present our main results, providing sufficient conditions for local criticality and local optimality for families of Steklov eigenvalues of either a spherical or annular domain in dimensions $d\geq 2$.   Numerical approaches based on MPS to solve Steklov eigenvalues are presented in Section 6 and obtained numerical results are consistent with theoretical studies. Section 7 concludes the findings and discusses future work.


\section{Steklov Eigenvalues and Eigenfunctions on A Ball and An Annular Domain}\label{SphericalHarmonicExamples}

In this section we consider examples of  Steklov eigenvalues and eigenfunctions for  spherical and annular domains.   We also offer a brief overview of spherical harmonics, highlighting some properties required for the proof of our main results.

\begin{example}\label{E:dimdBallEigs} For each $l = 0,1,2,\cdots$, the Steklov problem on $\mathbb{B}^d_{r_o}$, the $d$-dimensional ball with radius $r_o$, has an eigenvalue $\sigma =\frac{l}{r_o}$  which repeats with multiplicity $N_{l,d}$
\begin{equation*}
0,\ \cdots \ \underbrace{\frac{l}{r_o},\cdots \frac{l}{r_o}}_{N _{l,d}},\ \cdots
\end{equation*}

\noindent where, according to \cite[formula 2.10]{atkinson2012spherical}, we have

\begin{equation*}
\begin{Small}\text{$
N_{l,d} = \left(\begin{array}{c}
d+l-1\\
d-1
\end{array}\right)-\left(\begin{array}{c}
d+l-3\\
d-1
\end{array}\right)=\frac{(d+2l-2)(d+l-3)!}{l!(d-2)!}.
$}\end{Small}
\end{equation*}

\noindent Furthermore, the eigenspace of $\sigma = \frac{l}{r_o}$ has a basis of eigenfunctions 

\begin{equation*}
\begin{Small}\text{$
u_l^m (r,\theta_1,\cdots,\theta_{d-1}) \ =\underbrace{ \ r_o^{-\frac{d-1}{2}}\left( \frac{r}{r_o}\right)^l }_{N(r,l,d)}Y_l^m(\theta_1,\cdots,\theta_{d-1})  \quad \quad m=1,\cdots,N_{l,d},  
$}\end{Small}
\end{equation*}

\noindent where  $\{Y_l^m:\ m = 1,\cdots,N_{l,d}\}$  is an arbitrary orthonormal  basis for the $d$-dimensional spherical harmonics of degree $l$, see  \cite[section 2.1]{atkinson2012spherical};  and we write $N(r,l,d)$ for the radial dependence $r$ of the dimension $d$ eigenfunctions of $\sigma = \frac{l}{r_o}$  orthonormalized on the boundary $\mathbb{S}^{d-1}_{r_0}$. 

 In Figure \ref{fig: dim3DiskEigFun} we show an internal view of eigenfunctions corresponding to the first four non-zero distinct Steklov eigenvalues of the unit ball $\mathbb{B}^3$ in $\mathbb{R}^3$.

\begin{figure}[h!]
\centering\includegraphics[width = 5in]{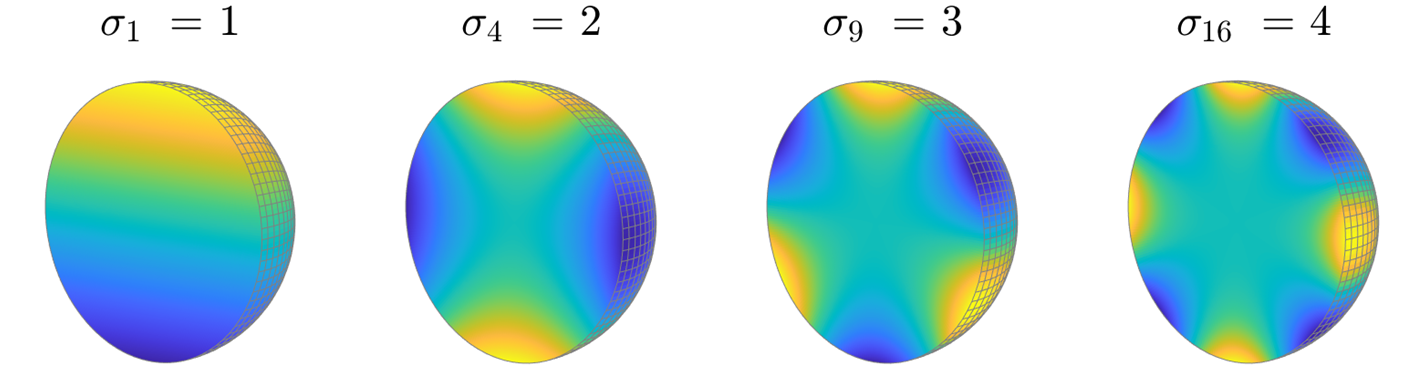}
\caption{Steklov eigenfunctions of the unit ball $\mathbb{B}^3$.}
\label{fig: dim3DiskEigFun}
\end{figure}
 
\end{example}

\begin{example}\label{E:ExampleAnnEigs}  As with spherical domains,  each eigenfunction of an annular domain can be written as a spherical harmonic multiplied by a radially dependent factor.  However, for annular domains the Steklov eigenvalues are no longer conveniently ordered  according to multiplicity. In general, because of the presence of a pair of independent boundary conditions, we have for each space of spherical harmonics $\mathbb{Y}_l^d$  that there  is  a pair of eigenvalues $\mu_{l,1}$ and $\mu_{l,2}$.  Where, for $k = 1,2$, the eigenspace of $\mu_{l,k}$ has  a basis $\{N(r_i,\mu_{l,k} ,d)Y_l^m$, $m =1,\cdots, N_{l,d}\}$. Here $Y_l^m$, $m =1,\cdots, N_{l,d}$ is an arbitrary basis for the  spherical harmonics of degree $l$, and each eigenfunction is boundary normalized by a radially dependent factor $ N(r_i,\mu_{l,k} ,d)$, which depends explicitly on the associated eigenvalue $\mu_{l,k}$. In Appendix \ref{AnnEigFormulas}, we provide general formulas for the eigenvalues and eigenfunctions of a $d$-dimensional annular domain $A^d_{r_i,r_o}$ with inner radius $r_i$ and outer radius $r_o$. 

In Figure \ref{fig: dim3AnnEigFun}, we show an internal view of a selection of eigenvalues and associated eigenfunctions of the three-dimensional annular domain $\mathbb{A}^3_{0.4,1}$ with inner radius $r_i = 0.4$ and outer radius $r_o = 1$.
\begin{figure}[h!]
\centering\includegraphics[width = 5in]{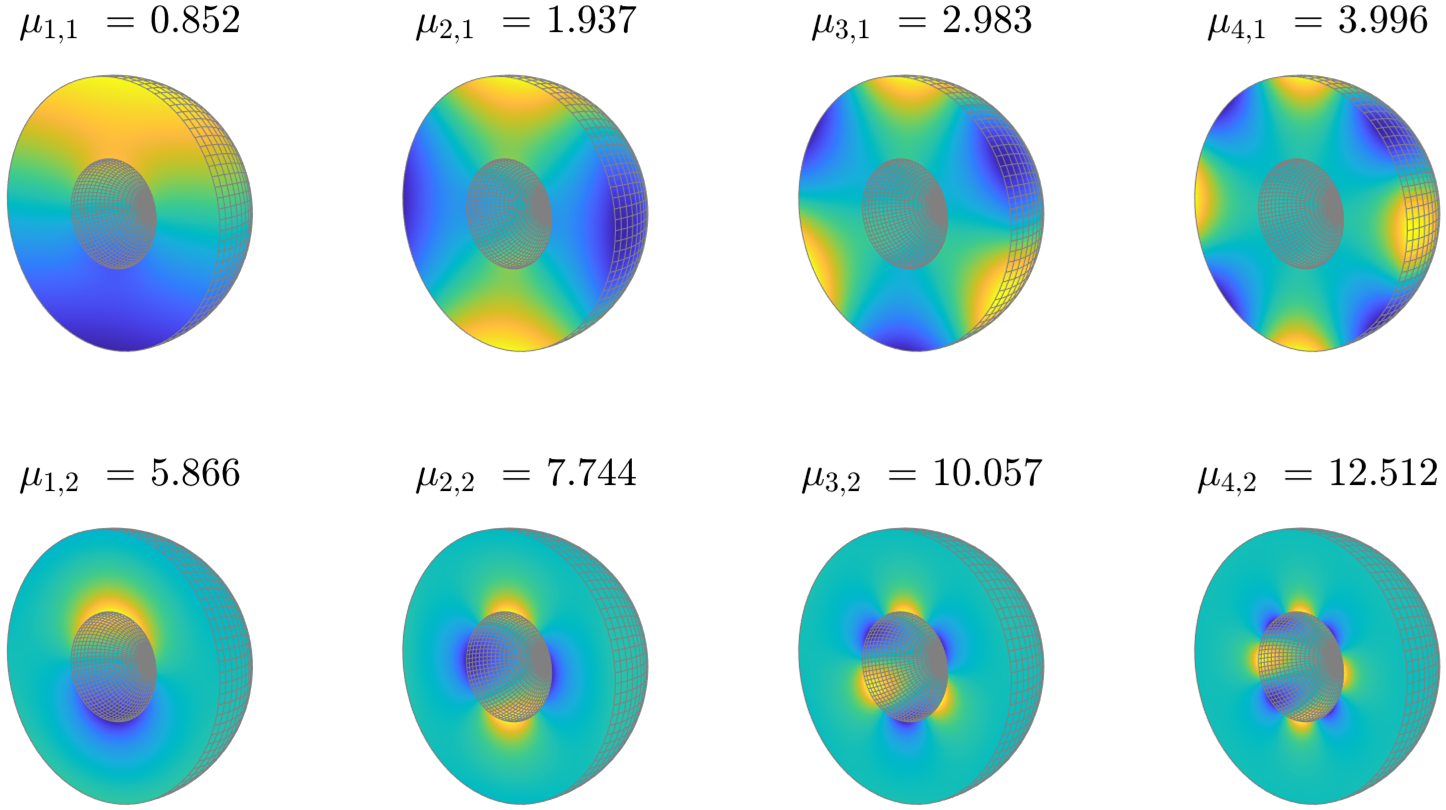}
\caption{Steklov eigenfunctions of the annular domain $\mathbb{A}^3_{0.4,1}$. }
\label{fig: dim3AnnEigFun}
\end{figure}

\end{example}

For both spherical and annular domains, spherical harmonics appear as a factor in the expressions for the Steklov eigenfunctions; and so, play an important role in the development that follows.  For this reason, we briefly recount some useful properties of spherical harmonics and refer readers to \cite[chapters 2 and 3]{atkinson2012spherical} for a detailed exposition. The vector space of  \textit{degree $l$ spherical harmonics $\mathbb{Y}_l^d$} is defined by restricting to the unit sphere $\mathbb{S}^{d-1}$, the polynomials on $\mathbb{R}^d$ that are both  harmonic and homogeneous of degree $l$  \cite[definition 2.7]{atkinson2012spherical}.  Where a polynomial $P:\mathbb{R}^d \to \mathbb{C}$ is harmonic if it satisfies Laplace's equation $\Delta P(\vec{x}) = 0$  for all $\vec{x} \in \mathbb{R}^d$; and homogeneous of degree $l$ if it satisfies $P(t\vec{x}) = t^l P(\vec{x})$ for all $t \in \mathbb{R}, \vec{x} \in \mathbb{R}^d$.  Here and below, we let  $\{Y_l^m:\ l = 0,1,\cdots;  m = 1,\cdots,N_{l,d}\}$  denote an arbitrary orthonormal  basis for the vector space of all $d$-dimensional spherical harmonics $\bigoplus_{l=0}^{\infty} \mathbb{Y}_l^d$.  Orthonormality is defined with respect to the Hilbert space $L^2(\mathbb{S}^{d-1})$, i.e.,
\begin{equation}\label{SHOrthonormal}
\begin{Small}\text{$ 
\int_{\mathbb{S}^{d-1}} Y_{l_1}^{m_1} \overline{  Y_{l_2}^ {m_2}}\ dS \ = \ \delta_{l_1,l_2} \delta_{m_1,m_2}.
$}\end{Small}
\end{equation}

\noindent Note that the spherical harmonics  $\bigoplus_{l=0}^{\infty} \mathbb{Y}_l^d$ are complete in $L^2(\mathbb{S}^{d-1})$, \cite[Theorem 2.38]{atkinson2012spherical}; and so  any $f \in L^2(\mathbb{S}^{d-1})$ has a \textit{Fourier-Laplace expansion in spherical harmonics}, i.e.,  for some choice of coefficients $\alpha_{l,m}  \in \mathbb{C}$ we have
\begin{equation}\label{FLExpansion}
\begin{Small}\text{$ 
f \   = 
 \  \sum_{l=0}^{\infty}\sum_{m=1}^{N_{l,d}} \alpha_{l,m}   Y_l^m.
$}\end{Small}
\end{equation}
 
It follows from the definitions that the degree 0 spherical harmonics $\mathbb{Y}_0^d$ consist of complex-valued constant polynomials; and the requirement that the basis be normalized in  $L^2(\mathbb{S}^{d-1})$ implies that for any basis we have $Y_0^1 = \frac{1}{\sqrt{\omega_{d-1}}}$, where \textit{$\omega_{d-1}$ is the surface area of the unit sphere $\mathbb{S}^{d-1}$}.  This observation allows us to evaluate the integral of any spherical harmonic basis element as follows.
\begin{prop}\label{SphericalHarmonicIntegral}
For any $l = 0,1,2,\cdots$ and $1\leq m \leq N_{l,d}$ we have that
\begin{equation}\label{2SHProdExp}
\begin{Small}\text{$ 
\int_{\mathbb{S}^{d-1}} Y_l^m \ dS \ = \ 
\begin{cases}
\sqrt{\omega_{d-1}} & \text{when}\  l = 0, \\
0 & \text{otherwise.}
\end{cases}
$}\end{Small}
\end{equation}
\end{prop}

\begin{proof}
If $l = 0$ then 
\begin{equation*}
\begin{Small}\text{$ 
\int_{\mathbb{S}^{d-1}} Y_0^1 \ dS \ = \  \int_{\mathbb{S}^{d-1}}  \frac{1}{\sqrt{\omega_{d-1}}} \ dS \ = \   \frac{\omega_{d-1}}{\sqrt{\omega_{d-1}}} \  = \ \sqrt{\omega_{d-1}},
$}\end{Small}
\end{equation*}

\noindent while if $l > 0$  we have  by orthonormality, see (\ref{SHOrthonormal}), that for any $m = 1,\cdots,N_{l,d}$ 
\begin{equation*} 
\begin{Small}\text{$ 
\int_{\mathbb{S}^{d-1}} Y_l^m  \ dS \ = \  \sqrt{\omega_{d-1}}\int_{\mathbb{S}^{d-1}} Y_0^1  Y_l^m  \ dS \ = \   0.
$}\end{Small}
\end{equation*}
\end{proof}

\noindent  It also follows from the definition that the product of any two spherical harmonics of degree $l$ is a homogeneous polynomial of degree $2l$; and so, in particular, we have the following expansion result.
\begin{prop}\label{2SphericalHarmonicExpansion}
The product of any two spherical harmonics in $\mathbb{Y}_l^d$ admits an expansion in spherical harmonics of even degree bounded above by $2l$, i.e., for all $1\leq m_1,m_2\leq N_{l,d}$ we have
\begin{equation}\label{2SHProdExp}
\begin{Small}\text{$ 
 Y_l^{m_1} \overline{  Y_l^ {m_2}} \ = \ \sum_{\substack{l^{\prime}=0 \\ l^{\prime} \ \text{even}}}^{2l}\sum_{m^{\prime}=1}^{N_{l^{\prime},d}}  c_{l^{\prime},m^{\prime}} Y_{l^{\prime}}^{m^{\prime}}
$}\end{Small}
\end{equation}
\noindent for some choice of coefficients $c_{l^{\prime}, m^{\prime}} \in \mathbb{C}$.
\end{prop}

\begin{proof}
See  \cite[Theorem 2.18 and Corollary 2.19]{atkinson2012spherical}.
\end{proof}

\noindent  We also make use of a special case of the Addition theorem, which  in three dimensions is  known as Unsöld's Theorem, and in the general case takes the following form.

\begin{prop}\label{AdditionTheoremSpecialCase}
Given any $l = 0,1,2,\cdots$ and any point $\vec{\theta} \in \mathbb{S}^{d-1}$ we have that 
\begin{equation}\label{AddTheoremSpclCase}
\begin{Small}\text{$ 
 \sum_{m = 1}^{N_{l,d}}  Y_l^m(\vec{\theta}) \overline{  Y_l^ m}(\vec{\theta}) \ = \ \frac{N_{l,d}}{\omega_{d-1}}.
$}\end{Small}
\end{equation}
\end{prop}

\begin{proof}
See  \cite[formula 2.35]{atkinson2012spherical}.
\end{proof}

In addition to appearing as Steklov eigenvalues of the unit ball $\mathbb{B}^d$,  the spherical harmonics also appear as eigenfunctions of the surface Laplacian $\Delta_{\tau}$ on the unit sphere $\mathbb{S}^{d-1}$, (see  \cite[section 3.3]{atkinson2012spherical}).  So, for any $l = 0,1,\cdots$ and any $m = 1,\cdots, N_{l,d}$, we have that the spherical harmonic $Y_ l^m \in \mathbb{Y}^d_l$ satisfies
\begin{equation} \label{SurfLaplaceEig} 
\begin{Small}\text{$ 
\Delta_{\tau} Y_l^m = - l(l+d-2) Y_l^m.
$}\end{Small}
\end{equation}

\noindent  We also have the following version of Greens Theorem on the unit sphere $\mathbb{S}^{d-1}$  (see  \cite[Proposition 3.3]{atkinson2012spherical}) .
\begin{prop}\label{GreenBeltramiIden}
(Green-Beltrami Identity) Given smooth functions $f :\mathbb{S}^{d-1} \to \mathbb{C}$ and  $g: \mathbb{S}^{d-1} \to \mathbb{C}$ we have
\begin{equation*}
\begin{Small}\text{$ 
\int_{\mathbb{S}^{d-1}} g \Delta_{\tau}f \ dS= -\int_{\mathbb{S}^{d-1}} \nabla_{\tau} g \cdot \nabla_{\tau}f dS.
$}\end{Small}
\end{equation*}
\end{prop}

\noindent   Combining the eigenvalue Formula (\ref{SurfLaplaceEig}) with Proposition \ref{GreenBeltramiIden}, we prove the following triple product integral identity for spherical harmonics, (see \cite{ongTripleProductIntegralIdentity}).  A result which plays an important role in Section \ref{dDformula}.

\begin{prop} \label{TripleProductIntIden}
(Triple Product Integral Identity) Given any spherical harmonic basis element of degree $l$ and any pair of spherical harmonic basis elements of degree $n$ we have
\begin{equation}
\begin{Small}\text{$ 
\int_{\mathbb{S}^{d-1}}  Y_l^m \nabla_{\tau} Y_n^i  \cdot \nabla_{\tau}\overline{  Y_n^j}   dS=  \left( n(n+d-2) - \frac{l(l+d-2)}{2} \right) \int_{\mathbb{S}^{d-1}} Y_l^m Y_n^i\overline{  Y_n^j}  dS.
$}\end{Small}
\end{equation}
\end{prop}

\begin{proof} Applying Formula (\ref{SurfLaplaceEig}), Proposition \ref{GreenBeltramiIden}, and the product rule for surface gradients, we have
\begin{align*}
\begin{Small}\text{$
 \int_{\mathbb{S}^{d-1}} Y_l^m \nabla_{\tau}   Y_n^i \cdot  \nabla_{\tau}\overline{  Y_n^j}dS
$}\end{Small} 
&=
\begin{Small}\text{$
 \int_{\mathbb{S}^{d-1}}\nabla_{\tau}  Y_n^i \cdot \nabla_{\tau}\left( \overline{  Y_n^j}Y_l^m\right)  dS -  \int_{\mathbb{S}^{d-1}} \overline{  Y_n^j}\nabla_{\tau}Y_n^i \cdot \nabla_{\tau}Y_l^m  dS
$}\end{Small}\\
 &= 
\begin{Small}
\text{$-\int_{\mathbb{S}^{d-1}}\overline{  Y_n^j}Y_l^m\Delta_{\tau}  Y_n^i \  dS -  \int_{\mathbb{S}^{d-1}}  \overline{  Y_n^j}\nabla_{\tau}Y_n^i \cdot \nabla_{\tau}Y_l^m  dS $}\end{Small}\\
  &=\begin{Small}
\text{$ n(n+d-2)\int_{\mathbb{S}^{d-1}}   Y_l^m  Y_n^i \overline{  Y_n^j}  dS -  \int_{\mathbb{S}^{d-1}} \overline{  Y_n^j} \nabla_{\tau}Y_n^i \cdot \nabla_{\tau}Y_l^m  dS $}\end{Small}\\
  &= \begin{Small}
\text{$n(n+d-2)\int_{\mathbb{S}^{d-1}}   Y_l^m Y_n^i\overline{  Y_n^j}    dS -  \int_{\mathbb{S}^{d-1}} Y_n^i \nabla_{\tau}\overline{  Y_n^j} \cdot \nabla_{\tau}Y_l^m dS.$}\end{Small}
 \end{align*}

\noindent Where the last equality follows from the fact that the previous three equalities are symmetric in $Y_n^i$ and $\overline{ Y_n^j}$.  From this last equality we deduce
\begin{equation*}
\begin{Small}\text{$
\int_{\mathbb{S}^{d-1}} Y_n^i \nabla_{\tau}\overline{  Y_n^j}\cdot \nabla_{\tau}Y_l^m dS = \int_{\mathbb{S}^{d-1}} \overline{  Y_n^j}  \nabla_{\tau}Y_n^i\cdot \nabla_{\tau}Y_l^m  dS.
$}\end{Small}
\end{equation*}

\noindent  A result which we apply to obtain the last equality in the derivation of the next formula.
\begin{align*}
\begin{Small}
\text{$\int_{\mathbb{S}^{d-1}} Y_n^i  \nabla_{\tau}\overline{  Y_n^j} \cdot \nabla_{\tau}Y_l^m\  dS $}\end{Small}
&= \begin{Small}
\text{$\int_{\mathbb{S}^{d-1}}\nabla_{\tau} \left( Y_n^i \overline{  Y_n^j}\right) \cdot \nabla_{\tau}  Y_l^m  dS -  \int_{\mathbb{S}^{d-1}}\overline{  Y_n^j} \nabla_{\tau} Y_n^i\cdot \nabla_{\tau}Y_l^m  dS $}\end{Small} \\
 &= \begin{Small}
\text{$-\int_{\mathbb{S}^{d-1}} Y_n^i \overline{  Y_n^j}  \Delta_{\tau}Y_l^m  dS -  \int_{\mathbb{S}^{d-1}}\overline{  Y_n^j} \nabla_{\tau} Y_n^i  \cdot \nabla_{\tau}Y_l^m  dS $}\end{Small}\\
 &=\begin{Small}
\text{$ l(l+d-2)\int_{\mathbb{S}^{d-1}}   Y_l^m Y_n^i \overline{  Y_n^j}     dS -  \int_{\mathbb{S}^{d-1}}\overline{  Y_n^j}  \nabla_{\tau} Y_n^i \cdot \nabla_{\tau}Y_l^m dS $}\end{Small}\\
 &= \begin{Small}
\text{$l(l+d-2)\int_{\mathbb{S}^{d-1}}   Y_l^m Y_n^i\overline{  Y_n^j}      dS -  \int_{\mathbb{S}^{d-1}} Y_n^i   \nabla_{\tau}\overline{  Y_n^j}\cdot \nabla_{\tau}Y_l^m  dS.$}\end{Small} 
 \end{align*}

\noindent It follows that
\begin{equation*}
\begin{Small}\text{$
\int_{\mathbb{S}^{d-1}}Y_n^i \nabla_{\tau}\overline{  Y_n^j}\cdot \nabla_{\tau}Y_l^m  dS =  \frac{l(l+d-2)}{2}\int_{\mathbb{S}^{d-1}}  Y_l^m Y_n^i\overline{  Y_n^j}   dS.
$}\end{Small}
\end{equation*}

 \noindent Combining this result with 
\begin{equation*}
\begin{Small}\text{$
 \int_{\mathbb{S}^{d-1}} Y_l^m  \nabla_{\tau} Y_n^i \cdot \nabla_{\tau}\overline{  Y_n^j}   dS  = n(n+d-2)\int_{\mathbb{S}^{d-1}}  Y_l^m Y_n^i \overline{  Y_n^j}   dS 
 -  \int_{\mathbb{S}^{d-1}}   Y_n^i\nabla_{\tau}\overline{  Y_n^j} \cdot \nabla_{\tau}Y_l^m  dS,
$}\end{Small}
\end{equation*}
 
 \noindent we have proved the result.
\end{proof}


\section{Mathematical Formulation}\label{MathForm}

We consider the problem of optimizing Steklov eigenvalue shape functionals among domains which are small perturbations of either a ball or an annular domain in $\mathbb{R}^d$.    In general, given a shape $\Omega$ and a deformation  field $V: \Omega \to \mathbb{R}^d$, we define for small $t \in \mathbb{R}$  the perturbed shape $\Omega_t = \{ x + tV(x) :  x \in \Omega \}$.    For a given shape functional $J$ and shape $\Omega$, the  \textit{local shape optimization problem} seeks conditions on a deformation field $V$ that guarantees $\Omega$ optimizes $J$ when restricted to the perturbation $\{\Omega_t\}$; i.e.,
\begin{equation*}
\begin{Small}\text{$
\Omega \ = \  \underset{{t \in (-\delta,\delta)}}{\text{argmax}}\  J\left(\Omega_t\right) \quad \text{or} \quad \Omega \ = \  \underset{{t \in (-\delta,\delta)}}{\text{argmin}}\  J\left(\Omega_t\right).
$}\end{Small}
\end{equation*}

\noindent  In this case we say that the pair \textit{$(\Omega, V)$ locally optimizes $J$}. In the event that for all $t \in (-\delta,\delta)$ we have $J(\Omega) > J(\Omega_t)$ or alternatively $J(\Omega) < J(\Omega_t )$, then we say  \textit{$(\Omega, V)$ locally strictly maximizes $J$} or  \textit{$(\Omega, V)$ locally striclty minimizes $J$} respectively.

  The function $J\left(\Omega,V\right)$ defined by the correspondence $t \mapsto J\left(\Omega_t\right)$  may be differentiable; and we write $dJ(\Omega,V)$ for its derivative at zero when it exists; i.e.,  
\begin{equation*}
\begin{Small}\text{$
 dJ(\Omega; V) = \underset{t\to 0}{\lim} \frac{J(\Omega_t) - J(\Omega)}{t}
$}\end{Small}
\end{equation*}

\noindent and say is $J$  \textit{locally differentiable at} $(\Omega,V)$. Clearly, by Fermat's theorem,  if $dJ(\Omega, V) \neq 0 $ then $(\Omega, V)$ cannot locally optimize $J$.  We say \textit{$(\Omega,V)$ is critical for   $J$} if  either  $dJ(\Omega; V) =0$ or $dJ(\Omega,V)$ is not defined.  In general, the determination of pairs $(\Omega, V)$   critical for $J$  is a natural way to start investigating  a  local  shape optimization problem.

Returning to the specific case  of a ball or an annular domain in $\mathbb{R}^d$, we  let $\mathbb{B}^d_{r_o}$ denote the $d$-dimensional open ball of radius $r_o$ and let $\mathbb{S}^{d-1}_{r_o}$ denote its boundary the {$(d-1)$-dimensional} sphere of radius $r_o$.  We also let $\mathbb{A}^d_{r_i,r_o}$ denote the $d$-dimensional annular domain with inner radius $r_i$ and outer radius $r_o$. 
\begin{align*}
\begin{Small}\text{$\mathbb{B}^d_{r_o}$}\end{Small} \ &=  \ \begin{Small}\text{$\{x\in \mathbb{R}^d: \Vert x \Vert < r_o\}$},\end{Small}\\
\begin{Small}\text{$\mathbb{S}^{d-1}_{r_o}$}\end{Small} \ &= \ \begin{Small}\text{$\{x \in \mathbb{R}^d : \Vert x \Vert = r_o\}$},\end{Small}\\
\begin{Small}\text{$\mathbb{A}^d_{r_i,r_o}$}\end{Small}\ &= \ \begin{Small}\text{$ \{ x \in \mathbb{R}^d : r_i < \Vert x \Vert < r_o\}$}.\end{Small}
\end{align*}
\noindent Unfortunately, when $\Omega$ is either a ball or an annular domain,  the multiplicity of the Steklov eigenvalue  $\sigma_n(\Omega)$ is greater than one and eigenvalue multiplicity is not preserved under perturbation.  As a consequence, depending on the deformation field,  the shape functionals $\sigma_n$ are not differentiable at 0, see \cite[section 2.5]{henrot2006extremum}. Therefore,  most pairs $(\Omega, V)$ are critical for $\sigma_n$ simply because $d\sigma_n(\Omega,V)$ does not exist. Despite this fact, the notion of local differentiability can still be used to investigate local optimization problems.  To explain how, we need some additional definitions.  Given a perturbation $\{\Omega_t\}$ generated by $(\Omega,V)$, we introduce the \textit{eigenvalue branch functions}  $\lambda_{0},\lambda_{1},\lambda_{2},\cdots $ defined by the rule that for each perturbation parameter $t$, we have that  $\lambda_{0}(t) \leq \lambda_{1}(t) \leq \lambda_{2}(t) \leq \cdots$ enumerate the Steklov eigenvalues of $\Omega_t$ repeated according to multiplicity, i.e. $\lambda_{k}(t) = \sigma_k(\Omega_t)$.   Next, we define the \textit{index} of an arbitrary Steklov eigenvalue $\sigma(\Omega)$  to be the smallest integer $n$ such that $\sigma(\Omega) = \sigma_n(\Omega)$.  It follows from the definitions that if $\sigma(\Omega)$ has index $n$ and multiplicity $p$ then $\lambda_n(0) = \lambda_{n+1}(0) = \cdots = \lambda_{n+p-1}(0) = \sigma$, in this case we call the functions $\lambda_n(t),\cdots,\lambda_{n+p-1}(t)$ the \textit{eigenvalue branches} of $\sigma(\Omega)$.  Note that  some or all of these eigenvalue branches may coincide.  

\begin{example}\label{E:UnitBallPetrubed}
The situation is illustrated in Figure \ref{fig: Y21} where the unit ball $\mathbb{B}^3$ is radially perturbed by the real spherical harmonic $Y_{2,1}$ according to the correspondence 
\begin{equation*}
\begin{Small}\text{$
V: \left[
\begin{matrix}
r \\
\theta\\
\end{matrix}
\right] 
\mapsto
\left[
\begin{matrix}
r \\
\theta\\
\end{matrix}
\right] +
\left[
\begin{matrix}
 rY_{2,1}(\theta))\\
0\\
\end{matrix}
\right].
$}\end{Small}
\end{equation*}

 \begin{figure}[h!]
 \centering\includegraphics[width = 3.5in]{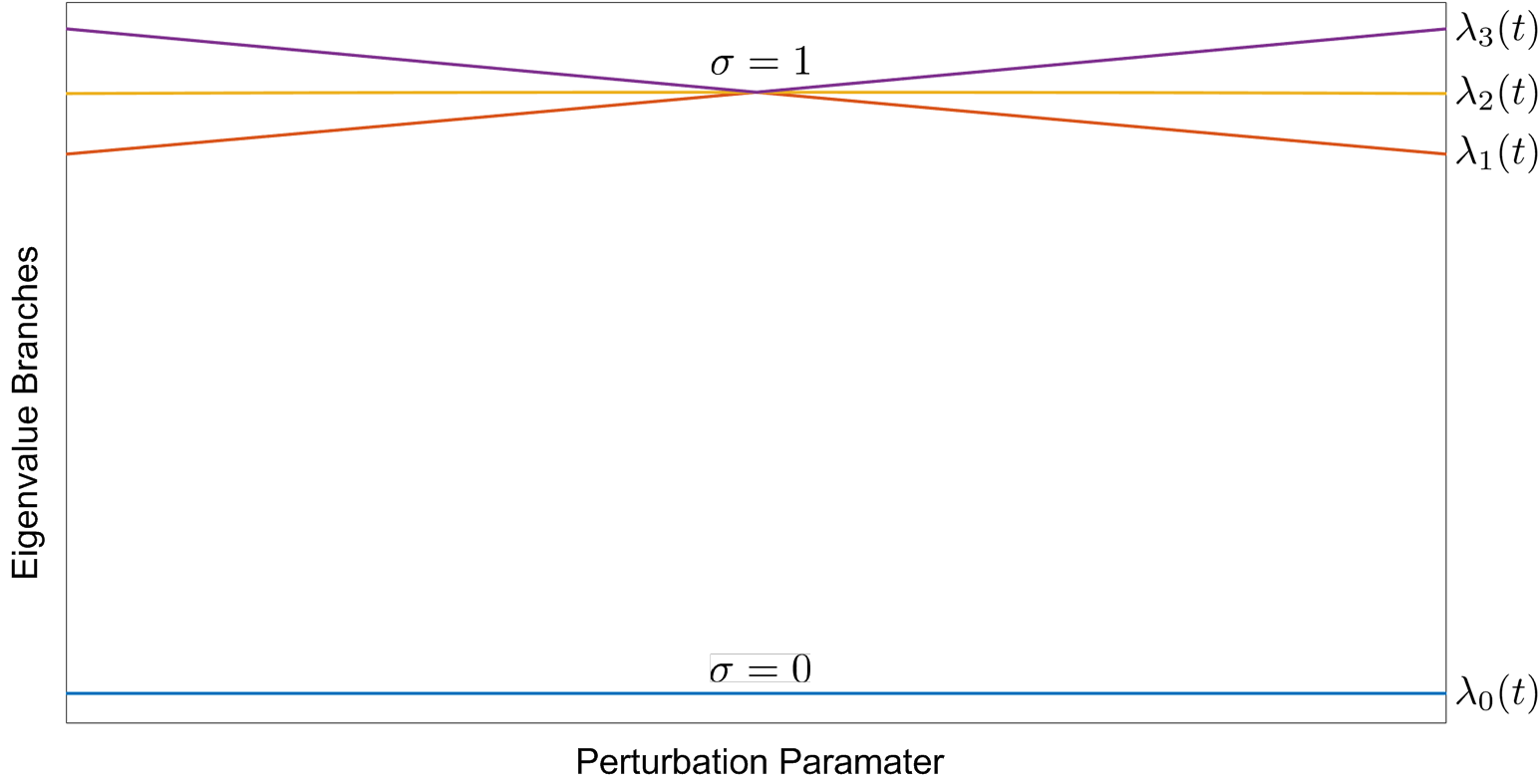}
  \caption{Eigenvalue branches of the unit ball $\mathbb{B}^3$.}
  \label{fig: Y21}
\end{figure}

In Figure \ref{fig: Y21}, $\sigma = 1$ has index 1 and  multiplicity 3 with eigenvalue branches $\lambda_1(t)$, $\lambda_2(t)$, $\lambda_3(t)$.  Also observe that for each perturbation parameter $t$ we have that  $\sigma_1(\Omega_t) = \lambda_1(t)$; and so, since  $\lambda_1(t)$ is not differentiable at 0,  we have an example where  $d\sigma_1(\mathbb{B}^3,V)$ fails to exist. 
Finally observe that there is a reordering of the branches on the left-hand side of the eigenvalue, $\sigma=1$ that will result in new branch functions $\tilde{\lambda}_k$ that are differentiable at 0.  For instance, if we let $s$ denote the permutation $1\mapsto 3, 2\mapsto 2, 3\mapsto 1$  and define for  $k=1,2,3$. 
\begin{equation*}
\begin{Small}\text{$
\tilde{\lambda}_{k}(t) =
\begin{cases}
\lambda_k (t) & t \geq 0,\\
\lambda_{s(k)}(t) & t < 0.
\end{cases}
$}\end{Small}
\end{equation*}
\noindent Then each of the functions $\tilde{\lambda}_{k}(t)$ is differentiable at 0.  Notice that  $\tilde{\lambda}_{1}^{\prime}(0)<0$ and  $\tilde{\lambda}_{3}^{\prime}(0)>0$, and so, from the fact that $\sigma_1(\Omega_t) = \lambda_1(t)$ and  $\sigma_3(\Omega_t) = \lambda_3(t)$ ,  we can infer that $(\mathbb{B}^3,V)$ locally strictly maximizes $\sigma_1$ and locally strictly minimizes $\sigma_3$.
\end{example}

 The next theorem shows that these observations are neither mistaken nor specific to this particular perturbation.  This result is a reformulation of  parts of Theorem E.1 of Dambrine-Kateb-Lamboley \cite{dambrine2014extremal}. We refer to this paper for an explanation of the notation.

\begin{theorem}\label{T:DKLE1}
\noindent Let $\Omega$ be bounded, open, and have Lipschitz boundary,  and let $V\in W^{3,\infty}(\Omega, \mathbb{R}^d)$. Suppose $\sigma$ is a Steklov eigenvalue of $\Omega$ with multiplicity $p$ and  index $n$.  Let $\left\{ u_i:\overline{\Omega} \to \mathbb{C}: i  =1,\cdots, p\right\}$, be a basis for the Steklov eigenspace of $\sigma$, orthonormalized with respect to the $L^2\left(\partial \Omega \right)$ inner product. Then we have

\noindent $\bullet$ If $\lambda_n(t),\cdots,\lambda_{n+p-1}(t)$ are the eigenvalue branches of $\sigma$, then  then there exists a permutation $s$ of ${\{n,\cdots,n+p-1\}}$ such that    for $k =n ,\cdots, n+p-1$ the shape functional
 
\begin{equation*}
\begin{Small}\text{$
\tilde{\lambda}_{k}(t) =
\begin{cases}
\lambda_{k}(t) & t \geq 0, \\
\lambda_{s(k)}(t) & t < 0,
\end{cases}
$}\end{Small}
\end{equation*}
is differentiable at 0; i.e.  $d\tilde{\lambda}_{k}(\Omega;V)$ exists.\\
\noindent $\bullet$  The derivatives $d\tilde{\lambda}_{k}(\Omega;V), \quad k=n,\cdots, n+p-1$ are equal to the eigenvalues of the $p\times p$ matrix $M = M (\Omega, V, \sigma)$ with entries defined by
\begin{equation*}
\begin{Small}\text{$
M_{i j} = \int_{\partial \Omega} \left(\nabla_{\tau}u_i
\cdot \nabla_{\tau}\overline{u_j} - \left( \sigma^2 + \sigma H\right) u_i\overline{u_j} \right)V_n\ dS.
$}\end{Small}
\end{equation*}
\noindent  Here $H$ is the additive curvature, i.e., the sum of the principal curvatures; and $V_n$ is the normal component of the deformation field $V$ on the boundary.  We refer to $M$ as an eigenvalue multiplicity perturbation matrix   or   \textit{EMP} matrix for short.
\quad\\
\end{theorem}


Based on Theorem \ref{T:DKLE1}, we  introduce the \textit{sub-differential} $\partial \sigma[\Omega,V]$ of the eigenvalue $\sigma(\Omega)$ as follows
\begin{equation*}
\begin{Small}\text{$
\partial \sigma[\Omega,V] = \left[\min_{k \in \{n,\cdots,n+p-1\}} \tilde{\lambda}_{k}^{\prime}[\Omega,V](0)\  ,\ \max_{k \in \{n,\cdots,n+p-1\}} \tilde{\lambda}_{k}^{\prime}[\Omega,V](0)\right].
$}\end{Small}
\end{equation*}

\noindent Notice that if $\sigma(\Omega)$ has index $n$ and multiplicity $p$  and $0 \notin \partial \sigma[\Omega,V]$ then $(\Omega,V)$ cannot locally optimize any of the shape functionals $\sigma_n,\cdots,\sigma_{n+p-1}$.  On the other hand if $0 \in \partial\sigma[\Omega,V]$ and $\min_{k} \tilde{\lambda}_{k}^{\prime}[\Omega,V](0) <0$ while $\max_{k} \tilde{\lambda}_{k}^{\prime}[\Omega,V](0)>0$ then as in Example \ref{E:UnitBallPetrubed}, we may conclude that $(\Omega,V)$ locally strictly maximizes $\sigma_n$ and locally strictly minimizes $\sigma_{n+p-1}$.   Based on these observations we define $(\Omega,V)$ to be \textit{critical for} $\sigma(\Omega)$ provided $0 \in \partial\sigma[\Omega,V]$.   Finally, because the trace of a matrix is the sum of its eigenvalues,  we have that if the  trace of the EMP matrix $ M (\Omega, V, \sigma)$ is zero, then we must have that $(\Omega,V)$ is critical for $\sigma(\Omega)$.   In what follows we seek a condition on $V$ which implies that the $trace[M (\Omega, V, \sigma)] =0$,  which in turn implies that $(\Omega,V)$ is critical for $\sigma(\Omega)$.


\section{ Nearly Spherical and Nearly Annular Domains in $\mathbb{R}^d$}\label{dDformula}

   In this section we show that for nearly spherical or nearly annular perturbations, we may express the entries of an EMP matrix as a finite sum of integrals of triple products of spherical harmonics.  We first prove the result for spherical domains.


\begin{theorem}\label{T:SphericalTripleProduct}
 Given a deformation field $V \in W^{3,\infty}(\mathbb{B}^d_{r_o}, \mathbb{R}^d)$,   the entries of the EMP matrix $M = M(\mathbb{B}^d_{r_o},V,\sigma)$ for the  Steklov eigenvalue $\sigma = \frac{n}{r_o}$ of $\mathbb{B}^d_{r_o}$ are given by 
\begin{equation*}
\begin{Small}\text{$ 
M_{i,j} =  \sum_{\substack{l=0 \\ l \ \text{even}}}^{2n}\sum_{m=1}^{N_{l,d}}  \alpha_{l,m,r_o}B(l,r_o) \int_{\mathbb{S}^{d-1} }  Y_n^i \overline{  Y_n^j} Y_l^m dS  
$}\end{Small}
\end{equation*}
\noindent where
\begin{equation*}
\begin{Small}\text{$ 
B(l,r)=  -\frac{1}{r^2}\left( \frac{l(l+d-2)}{2} + n \right)
$}\end{Small} 
\end{equation*}
\noindent The $\alpha_{l,m,r_o}$ are the coefficients of the Laplace-Fourier expansion of $V_{n,r_o}$ in the orthonormal basis $\{Y_l^m: l\geq 0; m= 1,\cdots,N_{l,d}  \}$ of $d$-dimensional spherical harmonics on the unit sphere $\mathbb{S}^{d-1}$.\\
\end{theorem}

\begin{proof}  According to  Example \ref{E:dimdBallEigs} and Theorem \ref{T:DKLE1}, an entry of the EMP matrix of a spherical domain is given by
\begin{equation}\label{SphereEMPInitial}
\begin{Small}\text{$
M(B^d_{r_o},V,\sigma)_{ij} =
 N(r_o,n,d)^2 \int_{\mathbb{S}^{d-1}_{r_o}} \left( \nabla_{\tau}Y_n^i \cdot \nabla_{\tau} \overline{ Y_n^j} - (\sigma^2 + H_{r_o}\sigma) Y_n^i \overline{Y_n^j} \right) V_{n,r_o}dS ,
$}\end{Small}
\end{equation}

\noindent where we  write  $V_{n,r_o}$ for the normal component of $V$ on $\mathbb{S}^{d-1}_{r_o}$; and we write $H_{r_o}$ for the additive curvature on  $\mathbb{S}^{d-1}_{r_o}$ .  Our goal is to simplify this integral expression.
 We start by expanding $V_{n,r_o}$ into a  Fourier-Laplace series as in Formula (\ref{FLExpansion}).  Next, we substitute and  reduce the integration to the unit sphere by making use of the fact that  $\int_{\mathbb{S}^{d-1}_{r_o}} dS = r_o^{d-1}\int_{\mathbb{S}^{d-1}} dS$  together with the fact that the surface gradient $\nabla_{\tau}$ evaluated on $\mathbb{S}^{d-1}_{r_o}$  equals  $\frac{1}{r_o} \nabla_{\tau}$ evaluated on $\mathbb{S}^{d-1}$.
\begin{align*}
&\begin{Small}\text{$ 
  \int_{\mathbb{S}^{d-1}_{r_o}}  \left( \nabla_{\tau}Y_n^i \cdot \nabla_{\tau} \overline{ Y_n^j}- (\sigma^2 + H_{r_o}\sigma) Y_n^i \overline{Y_n^j} \right) V_{n,r_o}dS
$}\end{Small}\\
& \   = 
\begin{Small}\text{$ 
 \sum_{l=0}^{\infty}\sum_{m=1}^{N_{l,d}} \alpha_{l,m,r_o} r_o^{d-1}
\left( \tfrac{1}{r_o^2}\int_{\mathbb{S}^{d-1}}  \nabla_{\tau}Y_n^i \cdot \nabla_{\tau} \overline{ Y_n^j} Y_l^m  dS - \left(\sigma^2 + H_{r_o}\sigma\right)  \int_{\mathbb{S}^{d-1}}  Y_n^i \overline{Y_n^j} Y_l^m dS \right) .
$}\end{Small}
\end{align*}

\noindent  Applying Proposition  \ref{TripleProductIntIden} to the first integral we have
\begin{align*}
&\begin{Small}\text{$ 
  \int_{\mathbb{S}^{d-1}_{r_o}}  \left( \nabla_{\tau}Y_n^i \cdot \nabla_{\tau} \overline{ Y_n^j}- (\sigma^2 + H\sigma) Y_n^i \overline{Y_n^j} \right) V_{n,r_o} dS
$}\end{Small}\\
&\ =
\begin{Small}\text{$ 
 \sum_{l=0}^{\infty}\sum_{m=1}^{N_{l,d}} \alpha_{l,m,r_o} r_o^{d-1}
\left( \tfrac{1}{r_o^2}\left( n(n+d-2) - \tfrac{l(l+d-2)}{2} \right) - \left(\sigma^2 + H_{r_o}\sigma\right)  \right)  \int_{\mathbb{S}^{d-1}}  Y_n^i \overline{Y_n^j} Y_l^m dS.
$}\end{Small}
\end{align*}

\noindent The sphere $\mathbb{S}^{d-1}_{r_o}$ has  additive curvature $H_{r_o} = \frac{d-1}{r_o}$ and the Steklov eigenvalue $\sigma = \frac{n}{r_o}$.  So, substituting and simplifying,  we find that 

\begin{equation*}
\begin{Small}\text{$
\frac{1}{r_o^2}\left( n(n+d-2) - \frac{l(l+d-2)}{2} \right) - \left(\sigma^2 +H\sigma\right) =  \frac{1}{r_o^2}\left(- \frac{l(l+d-2)}{2} - n \right).
$}\end{Small}
 \end{equation*}

\noindent Also, from  Example \ref{E:dimdBallEigs},  we have $N(r_o,n,d)^2 = r_o^{-(d-1)}$, and so plugging into \ref{SphereEMPInitial} and simplifying we obtain

\begin{equation*}
\begin{Small}\text{$ 
M_{i,j} = \sum_{l=0}^{\infty}\sum_{m=1}^{N_{l,d}}  \alpha_{l,m,r_o} \left[\frac{1}{r_o^2}\left(- \frac{l(l+d-2)}{2} - n \right)\right]\int_{\mathbb{S}^{d-1} }  Y_n^i \overline{  Y_n^j} Y_l^m dS .
$}\end{Small}
\end{equation*}

\noindent To show the sum is finite, we apply Proposition \ref{2SphericalHarmonicExpansion} to the product $Y_n^i \overline{  Y_n^j}$ to obtain the following expression for the triple product integral 

\begin{equation*}
\begin{Small}\text{$ 
\int_{\mathbb{S}^{d-1} }  Y_n^i \overline{  Y_n^j} Y_l^m dS  =  \sum_{\substack{l^{\prime}=0 \\ l^{\prime} \ \text{even}}}^{2n}\sum_{m^{\prime}=1}^{N_{l^{\prime},d}}  c_{l^{\prime},m^{\prime}} \int_{\mathbb{S}^{d-1} }  Y_{l^{\prime}}^{m^{\prime}} Y_l^m dS.
$}\end{Small}
\end{equation*}

\noindent It follows by orthonormality, see Formula  (\ref{SHOrthonormal}) , that the integral is guaranteed to be 0 when $l$ is odd and also for all $l > 2n$.  The result follows.
 
\end{proof}


Next, we prove the corresponding result for annular domains.

\begin{theorem}\label{T:AnnularTripleProduct}
 Given a deformation field $V \in W^{3,\infty}(\mathbb{A}^d_{r_ir_o}, \mathbb{R}^d)$,   the entries of the EMP matrix $M = M(\mathbb{A}^d_{r_ir_o},V,\mu_{n,k})$ for the Steklov eigenvalue $\mu_{n,k}$ of $ \mathbb{A}^d_{r_ir_o}$ are given by 
\begin{align*}
\begin{Small}\text{$ 
M_{i,j}$}\end{Small}
 \ =  \quad  
& \begin{Small}
\text{$  \sum_{\substack{l=0 \\ l \ \text{even}}}^{2n}\sum_{m=1}^{N_{l,d}}  \alpha_{l,m,r_o}A(r_o,l,\mu_{n,k} ,d) \int_{\mathbb{S}^{d-1} } Y_n^i \overline{  Y_n^j} Y_l^m  dS 
$}\end{Small}\\
 \ - \  &\begin{Small}
\text{$ \sum_{\substack{l=0 \\ l \ \text{even}}}^{2n}\sum_{m=1}^{N_{l,d}}  \alpha_{l,m,r_i}A(r_i,l,\mu_{n,k} ,d) \int_{\mathbb{S}^{d-1} }  Y_n^i \overline{  Y_n^j} Y_l^m dS
$}\end{Small}
\end{align*}
\noindent where
\begin{align*}
\begin{Small}\text{$
A(r_o,l,\mu_{n,k} ,d) =  N(r_o,\mu_{n,k} ,d)^2 \ r_o^{d-3}\left( n(n+d-2) - \frac{l(l+d-2)}{2} - r_o^2\mu^2_{n,k} -  (d-1)r_o\mu_{n,k}  \right)
$}\end{Small}
\end{align*}
\noindent and
\begin{align*}
\begin{Small}\text{$
A(r_i,l,\mu_{n,k} ,d) =  N(r_i,\mu_{n,k} ,d)^2 \ r_i^{d-3}\left( n(n+d-2) - \frac{l(l+d-2)}{2} - r_i^2\mu^2_{n,k} +  (d-1)r_i\mu_{n,k} \right) 
$}\end{Small}
\end{align*}
\noindent The  $\alpha_{l,m,r_o}$ (respectively the $\alpha_{l,m,r_i}$ ) are the coefficients of the Laplace-Fourier expansion  of $V_{n,r_o}$  (respectively $V_{n,r_i}$)  in the orthonormal basis $\{Y_l^m: l\geq 0; m = 1,\cdots,N_{l,d} \}$ of $d$-dimensional spherical harmonics on the unit sphere $\mathbb{S}^{d-1}$.  The radial normalization coefficient $N(r,\mu_{n,k},d)$ is defined in Appendix \ref{AnnEigFormulas}.\\
\end{theorem}

\begin{proof} According to Theorem \ref{T:DKLE1}, an entry of the EMP matrix of an annular domain is given by
\begin{align*}\label{AnnEMPInitial}
\begin{Small}\text{$ 
M_{ij} = \ N(r_o,n,k,d)^2 
$}\end{Small}
 & 
\begin{Small}\text{$ 
\int_{\mathbb{S}^{ d-1}_{r_o}} \left(\nabla_{\tau}Y_n^i \cdot \nabla_{\tau} \overline{ Y_n^j} - (\mu_{n,k}^2 + H_{r_o}\mu_{n,k}) Y_n^i \overline{Y_n^j} \right) V_{n,r_o} dS 
$}\end{Small}\\
\begin{Small}\text{$
 - \ N(r_i,n,k,d)^2 
$}\end{Small}
&
\begin{Small}\text{$
\int_{\mathbb{S}^{d-1}_{r_i}} \left( \nabla_{\tau}Y_n^i \cdot \nabla_{\tau} \overline{ Y_n^j} - (\mu_{n,k}^2 + H_{r_i}\mu_{n,k}) Y_n^i \overline{Y_n^j} \right) V_{n,r_i}dS 
$}\end{Small}. \nonumber
\end{align*}

\noindent We have that the additive curvatures on the outer and inner boundary are given by $H_{r_o} =  \frac{d-1}{r_o}$ and $H_{r_i} =  -\frac{d-1}{r_i}$.  So, the result follows immediately by applying the derivation in the proof of Theorem \ref{T:SphericalTripleProduct} to the individual terms of the EMP matrix $M$.   
\end{proof}



\section{The Local Steklov Eigenvalue Optimization Problem}

In this section we return to the local shape optimization problem for Steklov eigenvalue functionals and provide partial solutions for spherical and annular domains.  We first consider spherical domains; and start by calculating the trace of an EMP matrix.

\begin{prop}\label{TraceCalculationSphere}
Let $V \in W^{3,\infty}(\mathbb{B}^d_{r_o}, \mathbb{R}^d)$ and $\sigma = \frac{n}{r_o}$ be a Steklov eigenvalue of $\mathbb{B}^d_{r_o}$, then we have for the EMP matrix $M = M(\mathbb{B}^d_{r_o},V,\sigma)$ 
\begin{equation}\label{TraceSphere}
\begin{Small}\text{$ 
tr\left(M\right) \ = \  - \alpha_{0,1,r_o}\ \frac{n  N_{n,d}}{r_o\sqrt{\omega_{d-1}}}.
$}\end{Small}
\end{equation}
\end{prop}

\begin{proof}
If first we apply the result of Theorem \ref{T:SphericalTripleProduct};  and second, we  exchange the order of summation and integration; and third, we apply Proposition \ref{AdditionTheoremSpecialCase}; and finally, we apply Proposition \ref{SphericalHarmonicIntegral}, then we obtain
\begin{align*}
\begin{Small}\text{$tr\left(M\right) $}\end{Small} \ &= \
\begin{Small}\text{$
\sum_{j=1}^{N_{n,d}}\sum_{\substack{l=0 \\ l \ \text{even}}}^{2n}\sum_{m=1}^{N_{l,d}}  \alpha_{l,m,r_o}B(l,r_o) \int_{\mathbb{S}^{2}}Y_n^j \overline{ Y_n^j} Y_{l}^m  dS
$}\end{Small}\\
&= \ \begin{Small}\text{$
\sum_{\substack{l=0 \\ l \ \text{even}}}^{2n}\sum_{m=1}^{N_{l,d}}  \alpha_{l,m,r_o}B(l,r_o) \int_{\mathbb{S}^{2}}\left(\sum_{j=1}^{N_{n,d}}Y_n^j \overline{ Y_n^j}\right) Y_{l}^m  dS
$}\end{Small}\\
&= \ \begin{Small}\text{$
 \sum_{\substack{l=0 \\ l \ \text{even}}}^{2n}\sum_{m=1}^{N_{l,d}}  \alpha_{l,m,r_o}B(l,r_o)\frac{N_{n,d}}{\omega_{d-1}}  \int_{\mathbb{S}^{2}} Y_{l}^m  dS
$}\end{Small}\\
&= \
\begin{Small}\text{$
 \alpha_{0,1,r_o}B(0,r_o) \frac{N_{n,d}}{\sqrt{\omega_{d-1}}}.
$}\end{Small}
\end{align*}
\noindent Noting that $B(0,r_o) = -\frac{n}{r_o}$, the result follows.
\end{proof}

\noindent We now state our main result for spherical domains.

\begin{theorem}\label{T:LocalOptimizationBall}
Let $V \in W^{3,\infty}(\mathbb{B}^d_{r_o}, \mathbb{R}^d)$  be  a  boundary component volume preserving at first order deformation field on $\mathbb{B}^d_{r_o}$, i.e., $\int_{\mathbb{S}^{d-1}_{r_o}} V_n dS = 0$,  and let $\sigma =  \frac{n}{r_o}$  be a  Steklov eigenvalue  of $\mathbb{B}^d_{r_o}$, then $(\mathbb{B}^d_{r_o},V)$ is critical for $\sigma$.  If, in addition,  the EMP matrix $M = M(\mathbb{B}^d_{r_o},V,\sigma)$  is not the zero matrix, then $(\mathbb{B}_{r_o}^d,V)$ locally strictly maximizes $\sigma_{ind_{\sigma}}$ and locally strictly minimizes $\sigma_{ind_{\sigma}+N(n,d)-1}$.  Where $ind_{\sigma}$  is the index  of $\sigma$, while $N(n,d)$ is the multiplicity of $\sigma$.
\end{theorem}

\begin{proof}
Again applying Proposition \ref{SphericalHarmonicIntegral}, we have that 
\begin{align*}
\begin{Small}\text{$
\int_{\mathbb{S}_{r_0}^{d-1}} V_n\ dS \ = \
r_o^{d-1}\sum_{l=0}^{\infty}\sum_{m=1}^{N_{l,d}}\alpha_{m,l,r_o}  \int_{\mathbb{S}^{d-1}}Y_l^m\ dS 
\ = \ 
r_o^{d-1}\sqrt{\omega_{d-1}} \ \alpha_{0,1,r_o}  
$}\end{Small}
\end{align*}

\noindent and it follows that
\begin{align*}
\begin{Small}\text{ $V$ is volume preserving at first order on $\mathbb{B}^d_{r_o}
$}\end{Small}
\quad &\Leftrightarrow \quad
\begin{Small}\text{$
 \alpha_{0,1,r_o}  = 0 
$}\end{Small}\\
 \quad &\Leftrightarrow \quad
\begin{Small}\text{$
\text{trace}(M) = 0 
$}\end{Small}\\
\quad &\Rightarrow \quad
\begin{Small}\text{$
 0 \in \partial \sigma[\mathbb{B}_{r_o}^d,V] 
$}\end{Small}
\end{align*}

\noindent  We conclude $(\mathbb{B}_{r_o}^d,V)$  is critical for $\sigma$, (see the definition at the end of Section \ref{MathForm}).

Turning to the local optimization result, we assume that the EMP matrix $M$ has been computed using the standard orthonormal basis given by \ref{dDStdBasis} in Appendix \ref{StandardBasisSphHar}. For clarity, we assume the basis of $\mathbb{Y}_n^d$ has been enumerated in some way and retain the generic notation for these spherical harmonics while using the standard basis indices for the spherical harmonics used in the expansion of $V_{n,r_o}$.  Despite the notation it should be understood that all three spherical harmonics appearing in the integral are elements of the standard orthonormal basis. With this convention we have  
\begin{align*}
\begin{Small}\text{$
\overline{M}_{j,i} 
$}\end{Small}
\ &= \ 
\begin{Small}\text{$
\sum_{\mu_1 = 0}^{\infty}\ \sum_{{\mu_1,\cdots,\mu_{d-3},l,m}} \overline{\alpha}_{\mu_1,\cdots,\mu_{d-3},l,m,r_o}B(l,r_o) \int_{\mathbb{S}^{d-1} }  \overline{Y_n^j} \ \overline{ \overline{Y_n^i}}\  \overline{Y}_{\mu_1,\cdots,\mu_{d-3},l,m} dS 
$}\end{Small}\\
&= \ 
\begin{Small}\text{$
\sum_{\mu_1 = 0}^{\infty}\ \sum_{{\mu_1,\cdots,\mu_{d-3},l,m}} \alpha_{\mu_1,\cdots,\mu_{d-3},l,-m,r_o}B(l,r_o) \int_{\mathbb{S}^{d-1} }  \overline{Y_n^j}Y_n^i {Y}_{\mu_1,\cdots,\mu_{d-3},l,-m} dS
$}\end{Small}\\
&= \ 
\begin{Small}\text{$
\sum_{\mu_1 = 0}^{\infty}\ \sum_{{\mu_1,\cdots,\mu_{d-3},l,m}} \alpha_{\mu_1,\cdots,\mu_{d-3},l,m,r_o}B(l,r_o) \int_{\mathbb{S}^{d-1} }  Y_n^i \overline{Y_n^j}  {Y}_{\mu_1,\cdots,\mu_{d-3},l,m} dS \ = \ M_{ij}
$}\end{Small}
\end{align*}

\noindent Where for the second equality we have used Proposition \ref{ConjProperty} applied to $V_{n,r_o}$; and for the final equality we have rearranged the sum with respect to the symmetric index $m$ which ranges over $-l\leq m\leq l$.  We have shown the EMP matrix is Hermitian; and furthermore, because any pair of orthonormal bases for $\mathbb{Y}_n^d$ are unitarily equivalent, it is still the case that the EMP matrix is non-zero when computed with respect to the standard basis. 
 It follows that  if all the eigenvalues of $M$ were zero, then $M$ would have to be the zero matrix, a contradiction. We conclude that the EMP matrix $M$ must have a non-zero eigenvalue.  The conclusion that $(\mathbb{B}_{r_o}^d,V)$ locally maximizes $\sigma_{ind_{\sigma}}$ and locally minimizes $\sigma_{ind_{\sigma}+N(n,d)-1}$, follows from the discussion at the end of Section \ref{MathForm} together with the fact that $\text{trace}(M) = 0$.
\end{proof}

For an annular domain we have the following formula for the trace of an EMP matrix.   

\begin{prop}\label{TraceCalculationAnnulus}
Let $V \in W^{3,\infty}(\mathbb{A}^d_{r_i,r_o}, \mathbb{R}^d)$ and let $\mu_{l,k}$ be a Steklov eigenvalue of $\mathbb{B}^d_{r_o}$, then we have for the EMP matrix $M = M(\mathbb{A}^d_{r_i,r_o},V,\sigma)$ 
\begin{equation}\label{TraceSphere}
\begin{Small}\text{$ 
tr\left(M\right) \ = \ \left[\alpha_{0,1,r_o}A(r_o,0,\mu_{n,k} ,d) - \alpha_{0,1,r_i}A(r_i,0,\mu_{n,k} ,d) \right]\frac{N_{n,d}}{\sqrt{\omega_{d-1}}} .
$}\end{Small}
\end{equation}
\end{prop}

\begin{proof}
If first we apply the result of Theorem \ref{T:AnnularTripleProduct} and then separately apply the proof of Proposition \ref{TraceCalculationSphere} to the inner and outer expressions, we obtain the result.
\begin{align*}
\begin{Small}\text{$tr\left(M\right) $}\end{Small}
 \ =  \quad  
& \begin{Small}
\text{$ \sum_{j=1}^{N_{n,d}}  \sum_{\substack{l=0 \\ l \ \text{even}}}^{2n}\sum_{m=1}^{N_{l,d}}  \alpha_{l,m,r_o}A(r_o,l,\mu_{n,k} ,d) \int_{\mathbb{S}^{d-1} } Y_n^i \overline{  Y_n^j} Y_l^m  dS 
$}\end{Small}\\
 \ - \  &\begin{Small}
\text{$ \sum_{j=1}^{N_{n,d}} \sum_{\substack{l=0 \\ l \ \text{even}}}^{2n}\sum_{m=1}^{N_{l,d}}  \alpha_{l,m,r_i}A(r_i,l,\mu_{n,k} ,d) \int_{\mathbb{S}^{d-1} }  Y_n^i \overline{  Y_n^j} Y_l^m dS
$}\end{Small}\\
 = \quad &
\begin{Small}\text{$
\alpha_{0,1,r_o}A(r_o,0,\mu_{n,k} ,d) \frac{N_{n,d}}{\sqrt{\omega_{d-1}}} -  \alpha_{0,1,r_i}A(r_i,0,\mu_{n,k} ,d)\frac{N_{n,d}}{\sqrt{\omega_{d-1}}} .
$}\end{Small}
\end{align*}
\end{proof}

\noindent We also have the corresponding local optimization result for annular domains.
\begin{theorem}\label{T:LocalOptimizationAnnulus}
Let $V \in W^{3,\infty}(\mathbb{A}^d_{r_i,r_o}, \mathbb{R}^d)$ be  a  boundary component volume preserving at first order deformation field on $\mathbb{A}^d_{r_i,r_o}$, i.e., $\int_{\mathbb{S}^{d-1}_{r_o}} V_n dS = \int_{\mathbb{S}^{d-1}_{r_i}} V_n dS = 0$,  and let $\mu_{n,k}$  be a  Steklov eigenvalue  of $\mathbb{A}^d_{r_i,r_o}$,  then $(\mathbb{A}^d_{r_i,r_o},V)$ is critical for $\mu_{n,k}$.  If, in addition,  the EMP matrix $M = M(\mathbb{A}^d_{r_i,r_o},V,\mu_{n,k})$  is not the zero matrix, then $(\mathbb{A}^d_{r_i,r_o},V)$ locally maximizes $\sigma_{ind_{\mu}}$ and locally minimizes $\sigma_{ind_{\mu}+N(n,d)-1}$.  Where $ind_{\mu}$  is the index  of $\mu$, while $N(n,d)$ is the multiplicity of $\mu$. 
\end{theorem}

\begin{proof}
Arguing as in the proof of  Theorem \ref{T:LocalOptimizationBall}, we have that 
\begin{align*}
\begin{Small}\text{$
\int_{\mathbb{S}_{r_o}^{d-1}} V_{n,r_o}\ dS \ = \ r_o^{d-1}\sqrt{\omega_{d-1}} \ \alpha_{0,1,r_o}  \quad \text{and} \quad \int_{\mathbb{S}_{r_i}^{d-1}} V_{n,r_i}\ dS \ = \ r_i^{d-1}\sqrt{\omega_{d-1}} \ \alpha_{0,1,r_i}
$}\end{Small}
\end{align*}

\noindent and it follows that
\begin{align*}
\begin{Small}\text{ $V$ is boundary component volume preserving at first order on $\mathbb{A}^d_{r_i,r_o}
$}\end{Small}
\quad &\Leftrightarrow \quad
\begin{Small}\text{$
 \alpha_{0,1,r_o}  = 0 
$}\end{Small}\\
 \quad &\Rightarrow \quad
\begin{Small}\text{$
\text{trace}(M) = 0 
$}\end{Small}\\
\quad &\Rightarrow \quad
\begin{Small}\text{$
 0 \in \partial \sigma[ \mathbb{A}^d_{r_i,r_o},V] 
$}\end{Small}
\end{align*}

\noindent  Therefore $(\mathbb{A}^d_{r_i,r_o},V)$  is critical for $\sigma$.

Because  the sum of two Hermitian matrices is Hermitian, we show the EMP matrix is Hermitian by arguing as in the proof of Theorem \ref{T:LocalOptimizationBall}, that the inner and outer components of the EMP matrix are Hermitian.   The local optimization result follows by arguing as in the proof of Theorem \ref{T:LocalOptimizationBall}, that a non-zero and Hermitian EMP matrix must have a non-zero eigenvalue.
\end{proof}


\section{Numerical Implementation and Numerical Results}\label{Numerics}

\subsection{Method of Particular Solutions} Making use of the \textit{method of particular solutions}, we numerically investigate and illustrate Theorems \ref{T:LocalOptimizationBall} and  \ref{T:LocalOptimizationAnnulus}. We consider perturbations of spherical domains $\Omega = \mathbb{B}^d_{r_o}$ and annular domains $\Omega = \mathbb{A}^d_{r_i,r_o}$,  where for each small perturbation parameter $t\in \mathbb{R}$, we have 
\begin{equation*}
 \begin{Small}\text{$
\Omega_t = \left\{
\left[
\begin{matrix}
r\\
\vec{\theta}\\
\end{matrix}
\right] +
t \left[
\begin{matrix}
 rV(\vec{\theta})\\
0\\
\end{matrix}
\right]
: \ \left[
 \begin{matrix}
r \\
\vec{\theta}
\end{matrix}
\right] 
 \in \Omega
\right\}.
$}\end{Small}
\end{equation*}

\noindent We limit attention to deformation fields $V$ such that the normal components $V_{r_o}$ and $V_{r_i}$ can be written as a finite linear combination  of spherical harmonics taken from an arbitrary orthonormal basis  $\{Y_l^m:\ l = 0,1,\cdots;  m = 1,\cdots,N_{l,d}\}$. In this case,  a solution $u$ of the Steklov eigenvalue problem (\ref{E:StekProbGen}) on the perturbed  domain $\Omega_t$  may be expanded in a Fourier-Laplace series of regular and singular solid harmonics.  Indeed, if  $\Omega_t$ is nearly spherical, then we have an expansion in regular solid harmonics
\begin{equation*}
\begin{Small}\text{$ 
u(r,\vec{\theta})  \   =  \  \sum_{l=0}^{\infty}\sum_{m=1}^{N_{l,d}} a_{l,m} r^l  Y_l^m(\vec{\theta}),
$}\end{Small}
\end{equation*}
\noindent while if $\Omega_t$ is nearly annular and $d\geq 3$, then the expansion also includes singular solid harmonics
\begin{equation*}
\begin{Small}\text{$ 
u(r,\vec{\theta})  \   =  \  \sum_{l=0}^{\infty}\sum_{m=1}^{N_{l,d}} a_{l,m} r^l Y_l^m(\vec{\theta})) \  + \  \sum_{l=0}^{\infty}\sum_{m=1}^{N_{l,d}} b_{l,m} r^{-(d+l-2)} Y_l^m(\vec{\theta}).
$}\end{Small}
\end{equation*}
\noindent When $\Omega_t$ is nearly annular and $d=2$, care must be taken because the leading singular solid harmonic includes a logarithmic dependency on $r$. So, noting that $N_{l,2} = 2$ for all $l\geq 1$, we have
\begin{equation*}
\begin{Small}\text{$ 
u(r,\vec{\theta})  \   =  \ a_{0,1}Y^1_0 + b_{0,1}\log(r) Y_0^1 +  \sum_{l=1}^{\infty}\sum_{m=1}^2 a_{l,m} r^l Y_l^m(\vec{\theta})) \  + \  \sum_{l=1}^{\infty}\sum_{m=1}^2 b_{l,m} r^{-(d+l-2)} Y_l^m(\vec{\theta}).
$}\end{Small}
\end{equation*}
\noindent The natural logarithm occurs  because when the Steklov equation is solved on $\mathbb{A}^d_{r_i,r_o}$ by separation of variables, we find that the radially dependent factor must satisfy a Cauchy-Euler equation; and in two dimensions, the solution of this equation includes a logarithmic term when $l=0$.

Based on these Fourier-Laplace expansions, if we denote the regular solid harmonics by
\begin{equation*}
\begin{Small}\text{$ 
u^{r}_{l,m}(r,\vec{\theta}) \ = \ 
  r^l Y_l^m(\vec{\theta}),
$}\end{Small}
\end{equation*}
\noindent and the singular solid harmonics by
\begin{align*}
\begin{Small}\text{$ 
 u^s_{0,1}(r,\vec{\theta})
$}\end{Small}
 \   &=  \
\begin{Small}\text{$
\begin{cases} \log(r)Y_0^1 & d = 2,\\
 r^{-(d-2)}Y_0^1 & d\geq 3, 
\end{cases} 
$}\end{Small}\\
&\quad\\
\begin{Small}\text{$
u^{s}_{l,m}(r,\vec{\theta})
$}\end{Small}
\   &=  \ 
\begin{Small}\text{$
  r^{-(d + l+2)} Y_l^m(\vec{\theta}) \ \  \text{for}\ l \geq 1,
$}\end{Small}
\end{align*}
\noindent then, for a fixed  choice of maximum spherical harmonic degree $L$, we have the following approximate solution ansatz for the Steklov problem.  When  $\Omega_t$ is nearly spherical  
\begin{equation*}
\begin{Small}\text{$ 
u^L(r,\vec{\theta}) \   = 
 \  \sum_{l=0}^{L}\sum_{m=1}^{N_{l,d}} a_{l,m}u^{r}_{l,m}(r,\vec{\theta}),
$}\end{Small}
\end{equation*}
\noindent and when $\Omega_t$ is nearly annular 
\begin{equation*}
\begin{Small}\text{$ 
u^L(r,\vec{\theta}) \   = 
 \  \sum_{l=0}^{L}\sum_{m=1}^{N_{l,d}} a_{l,m}u^{r}_{l,m}(r,\vec{\theta}) + \sum_{l=0}^{L}\sum_{m=1}^{N_{l,d}} b_{l,m,} u^{s}_{l,m}(r,\vec{\theta}).
$}\end{Small}
\end{equation*}

\noindent In both cases $u^L$ clearly  satisfies Laplace's equation on $\Omega_t$, i.e., $\Delta u^L(r,\vec{\theta}) = 0$ for $(r,\vec{\theta}) \in \Omega_t$. Suppose for a given $\Omega_t$, we can find approximate eigenvalues $\sigma$ and corresponding  coefficients $a_{l,m}$ and $b_{l,m}$ such that  $u^L$ also satisfies the Steklov boundary condition  $\partial_n u^L(r,\vec{\theta})  = \sigma u^L(r,\vec{\theta})$  on some collection of points $(r,\vec{\theta})$ distributed across the boundary $\partial \Omega_t$.  Then letting the perturbation parameter $t$ vary over a  discrete set of values around and including zero and ordering the resulting approximate eigenvalues according to multiplicity, we produce a numerical approximation to the eigenvalue branches for $ \mathbb{A}^d_{r_i,r_o}$ as described in Theorem  \ref{T:DKLE1}.  To obtain the unknown eigenvalues $\sigma$ and corresponding  coefficients $a_{l,m}$ and $b_{l,m}$, we solve a generalized eigenvalue problem
\begin{equation}\label{E:GenEigProblem}
\begin{Small}\text{$ 
B^TA\vec{\alpha} = \sigma B^TB \vec{\alpha}
$}\end{Small}
\end{equation}
\noindent For clarity we describe the derivation of Equation (\ref{E:GenEigProblem}) for the unperturbed case when $t=0$, and later indicate the changes required in the general perturbed case of $\Omega_t$, with $t\neq 0$.

In the case when $\Omega = \mathbb{B}^d_{r_o}$ we have  $\vec{\alpha}= [a_{0,1},\cdots, a_{L,N_{L,d}}]^T$ is the column vector of unknown coefficients of $u^L$; and the matrices $A$ and $B$ are defined as follows.  Let $(r_{o}, \vec{\theta}_{o,k})_{k=1}^{K_o}$  be a collection of points distributed on the sphere $\mathbb{S}^{d-1}_{r_o}$ as described in \cite{deserno2004generate}.  If we define $A$ and $B$ to be the $K_o\times(\sum_{l=0} ^L N_{l,d})$ matrices whose $\text{k}^{th}$ rows are respectively given by  
\begin{align*}
\begin{Small} \text{$
A(k,:)
$}\end{Small}   \  &= \
\begin{Small} \text{$
\left [\partial_n u^{r}_{0,1}( r_o ,\vec{\theta}_{o,k}),\cdots,\partial_n u^{r}_{L,N_{L,d}}(r_o  ,\vec{\theta}_{o,n})\right],
$}\end{Small} \\
\begin{Small} \text{$
B(k,:)
$}\end{Small}   \  &= \
\begin{Small} \text{$
\left [u^{r}_{0,1}( r_o  ,\vec{\theta}_{o,k}),\cdots,u^{r}_{L,N_{L,d}}( r_o  ,\vec{\theta}_{o,k}) \right], 
$}\end{Small}   
\end{align*}
 then the equation $A\vec{\alpha} = \sigma B \vec{\alpha}$ expresses the requirement  that the Steklov eigenvalue equation $\partial u = \sigma  u$ is true on the boundary $\mathbb{S}^{d-1}_{r_o}$ at each of the  points    
$(r_o, \vec{\theta}_{o,k})_{k=1}^{K_o}$.   To prevent an ill-conditioned problem, we multiply both sides of this equation by $B^T$ and obtain our final generalized eigenvalue problem (\ref{E:GenEigProblem}).

When $\Omega = \mathbb{A}^{d}_{r_i,r_o}$ we have  $\vec{\alpha} = [a_{0,1},\cdots, a_{L,N_{L,d}},b_{0,1}, \cdots, b_{L,N_{L,d}}]^T$ is the column vector of unknown coefficients of $u^L$; and the matrices $A$ and $B$ are defined as follows.  Let $(r_{o}, \vec{\theta}_{o,k})_{k=1}^{K_o}$  and $(r_{i}, \vec{\theta}_{i,k})_{k=1}^{K_i}$ be collections of points distributed on the spheres $\mathbb{S}^{d-1}_{r_o}$ and $\mathbb{S}^{d-1}_{r_i}$ respectively.    Define $A_o, B_o$ to be the $K_o\times(\sum_{l=0} ^L N_{l,d})$ matrices and $ A_i, B_i$ to be the $K_i\times(\sum_{l=0} ^L N_{l,d})$ matrices whose $\text{k}^{th}$ rows are respectively given by  
\begin{align*}
\begin{Small} \text{$
A_o(k,:)
$}\end{Small}   \  &= \
\begin{Small} \text{$
\left [\partial_n u^{r}_{0,1}( r_o,\vec{\theta}_{o,k}),\cdots,\partial_n u^{r}_{L,N_{L,d}}(r_o,\vec{\theta}_{o,n}), \partial_n u^{s}_{0,1}(r_o,\vec{\theta}_{o,k}),\cdots,\partial_n u^{s}_{L,N_{L,d}}(r_o,\vec{\theta}_{o,k}) \right],
$}\end{Small} \\
\begin{Small} \text{$
B_o(k,:)
$}\end{Small}   \  &= \
\begin{Small} \text{$
\left [u^{r}_{0,1}(r_o,\vec{\theta}_{o,k}),\cdots,u^{r}_{L,N_{L,d}}(r_o,\vec{\theta}_{o,k}),  u^{s}_{0,1}(r_o,\vec{\theta}_{o,k}),\cdots, u^{s}_{L,N_{L,d}}(r_o,\vec{\theta}_{o,k}) \right], 
$}\end{Small}   \\
\begin{Small} \text{$
A_i(k,:)
$}\end{Small}  \  &= \
 \begin{Small} \text{$
\left [\partial_n u^{r}_{0,1}(r_i,\vec{\theta}_{i,k}),\cdots,\partial_n u^{r}_{L,N_{L,d}}(r_i,\vec{\theta}_{i,n}), \partial_n u^{s}_{0,1}(r_i,\vec{\theta}_{i,k}),\cdots,\partial_n u^{s}_{L,N_{L,d}}(r_i,\vec{\theta}_{i,k}) \right], 
$}\end{Small} \\
\begin{Small} \text{$
B_i(k,:)
$}\end{Small}   \  &= \
\begin{Small} \text{$
 \left[u^{r}_{0,1}(r_i,\vec{\theta}_{i,k}),\cdots,u^{r}_{L,N_{L,d}}(r_i,\vec{\theta}_{i,k}),  u^{s}_{0,1}(r_i,\vec{\theta}_{i,k}),\cdots, u^{s}_{L,N_{L,d}}(r_i,\vec{\theta}_{i,k}) \right], 
$}\end{Small}  
\end{align*}
If  we define 
\begin{equation*}
\begin{Small}\text{$
A = 
\left[
\begin{matrix}
A_o \\
A_i \\
\end{matrix}
\right]
\quad \quad \text{and} \quad \quad
B = 
\left[
\begin{matrix}
B_o \\
B_i 
\end{matrix}
\right],
$}\end{Small} 
\end{equation*}
\noindent then the equation $A \vec{\alpha} = \sigma B\vec{\alpha}$ expresses the requirement  that the Steklov eigenvalue equation $\partial u^L= \sigma u^L$ is true on both the outer and inner boundary at each of the  points    $(r_o, \vec{\theta}_{o,k})_{k=1}^{K_o}$  and $(r_i, \vec{\theta}_{i,k})_{k=1}^{K_i}$.   To prevent an ill-conditioned problem, we multiply both sides of this equation by $B^T$ and obtain our final generalized eigenvalue problem (\ref{E:GenEigProblem}).

For a perturbed domain $\Omega_t$, with $t \neq 0$,   the derivation of Equation (\ref{E:GenEigProblem}) requires only that all occurrences of the points  $(r_{o}, \vec{\theta}_{o,k})_{k=1}^{K_o}$ in the nearly spherical case, and   $(r_{o}, \vec{\theta}_{o,k})_{k=1}^{K_o}$ together with $(r_{i}, \vec{\theta}_{i,k})_{k=1}^{K_i}$  in the nearly annular case, be replaced with their images under the deformation field, i.e.,   $(r_{t,o,k}, \vec{\theta}_{o,k})_{k=1}^{K_o}$  and $(r_{t,i,k}, \vec{\theta}_{i,k})_{k=1}^{K_i}$, where $r_{t,o,k} = r_o + tr_oV(\vec{\theta}_{o,k})$ and  $r_{t,i,k} = r_i + tr_i V(\vec{\theta}_{i,k})$.  With this substitution the solutions of Equation (\ref{E:GenEigProblem}) provide approximations to the Steklov eigenvalues of $\Omega_t$.

\subsection{Numerical Results}
\begin{example}\label{E:Example2DNumerics}  In two dimensions we have that the multiplicity of any non-zero Steklov eigenvalue for either a disk or an annulus is given by  $N_{n,2} =  2$.  Therefore, the EMP matrices described in Theorems \ref{T:SphericalTripleProduct}   and \ref{T:AnnularTripleProduct} are $2\times2$ and so may be explicitly calculated. Indeed, with the following basis for the 2D spherical harmonics
\begin{align*}
\begin{Small}\text{$ 
Y_0^1
$}\end{Small} &= 
\begin{Small}\text{$ 
\frac{1}{\sqrt{2\pi}}\quad \quad \text{for}\  l = 0,
$}\end{Small}\\
\begin{Small}\text{$ 
Y_l^m 
$}\end{Small}
&=
\begin{Small}\text{$
  \frac{1}{\sqrt{2\pi}}e^{i(-1)^ml\theta} \quad \text{for}\  l\geq 0, \ m = 1,2,
$}\end{Small}
\end{align*}

\noindent we obtain the following two corollaries to Theorems \ref{T:SphericalTripleProduct}   and \ref{T:AnnularTripleProduct}  respectively.
\begin{corollary}\label{C:2DEMPMatrixBall}
 Given a deformation field $V \in W^{3,\infty}(\mathbb{B}^2_{r_o}, \mathbb{R}^2)$,   the EMP matrix $M(\mathbb{B}^2_{r_o},V,\sigma)$ for the  Steklov eigenvalue $\sigma =  \frac{n}{r_o}$ of $\mathbb{B}^2_{r_o}$ is  given by \\
\begin{equation*}
\begin{Small}\text{$ 
-\frac{n}{r_o^2 \sqrt{2\pi}}
\left[
\begin{matrix}
\alpha_{0,1,r_o} & (2n+1)\alpha_{2n,2,r_o}\\
(2n+1)\alpha_{2n,1,r_o} & \alpha_{0,1,r_o} \\
\end{matrix}
\right].
$}\end{Small}
\end{equation*}\\
\end{corollary}

\begin{corollary}\label{C:2DEMPMatrixAnnulus}
 Given a deformation field $V \in W^{3,\infty}(\mathbb{A}^2_{r_i,r_o}, \mathbb{R}^2)$,   the EMP matrix  $M(\mathbb{A}^2_{r_i,r_o},V,\mu_{n,k})$ for the  Steklov eigenvalue $\mu_{n,k}$ of $\mathbb{A}^2_{r_i,r_o}$ is  given by \\
\begin{equation*}
\begin{Small}\text{$ 
\frac{1}{\sqrt{2\pi}} 
\left[
\begin{matrix}
 A(r_o,0,\mu_{n,k} ,2)\ \alpha_{0,1,r_o} &  A(r_o,2n,\mu_{n,k} ,2)\ \alpha_{2n,2,r_o}\\
A(r_o,2n,\mu_{n,k} ,2)\ \alpha_{2n,1,r_o} &  A(r_o,0,\mu_{n,k} ,2)\ \alpha_{0,1,r_o} \\
\end{matrix}
\right] \ - \ 
\frac{1}{\sqrt{2\pi}} 
\left[
\begin{matrix}
A(r_i,0,\mu_{n,k} ,2)\ \alpha_{0,1,r_i} &  A(r_i,2n,\mu_{n,k} ,2)\ \alpha_{2n,2,r_i}\\
 A(r_i,2n,\mu_{n,k} ,2)\ \alpha_{2n,1,r_i} &  A(r_i,0,\mu_{n,k} ,2)\ \alpha_{0,1,r_i} \\
\end{matrix}
\right] ,
$}\end{Small}
\end{equation*}\\
\end{corollary}

Note that when $V$ is boundary component volume preserving at first order, as in Theorems \ref{T:LocalOptimizationBall} and  \ref{T:LocalOptimizationAnnulus}, we have that $\alpha_{0,1,r_o} = \alpha_{0,1,r_i} = 0$.  It follows that given a Steklov eigenvalue $\sigma = \frac{l}{r_o}$ of a disk or $\mu_{n,k}$ of an annulus, a  necessary condition for the EMP matrix to be non-zero is that at least one of 
$\alpha_{2n,1,r_o}$, $\alpha_{2n,2,r_o}$ respectively $\alpha_{2n,1,r_o}$, $\alpha_{2n,1,r_i}$, $\alpha_{2n,2,r_o}$, $\alpha_{2n,2,r_i}$ be non-zero.  For a disk this condition is also sufficient.  Interestingly, for a 2D annulus it is possible for the EMP matrix to be zero even when all of $\alpha_{2n,1,r_o}$, $\alpha_{2n,1,r_i}$, $\alpha_{2n,2,r_o}$, $\alpha_{2n,2,r_i}$ are non-zero.  Indeed, this will be the case provided the Fourier coefficients are selected non-zero and satisfying 
\begin{equation}\label{RatioCondition}
\frac{\alpha_{2n,1,r_o}}{\alpha_{2n,1,r_i}} = \frac{\alpha_{2n,2,r_o}}{\alpha_{2n,2,r_i}} = \frac{A(r_i,2n,\mu_{n,k} ,2)}{A(r_o,2n,\mu_{n,k} ,2)}.
\end{equation}

In Figure \ref{fig:2DBowtieResponse}(a) we visualize some eigenvalue branches and their tangent lines at zero when $\mathbb{A}^2_{0.4,1}$ is perturbed by a deformation field $V$ with $V_{n,r_o} = V_{n,r_i} = 2\cos(6\theta)$ In this case, condition (\ref{RatioCondition}) is not satisfied for $n=3$,  and the EMP matrices of $\mu_{3,1} = 2.944$ and $\mu_{3,2} = 7.642$ are non-zero.  We see that the branches of these two eigenvalues demonstrate the characteristic "bow tie" response to the perturbation.  The EMP matrices of all other eigenvalues are zero, and their branches  all have slope zero at the initial shape. In Figure \ref{fig:2DBowtieResponse}(b) the same domain has been perturbed by a deformation field $V$ with  $V_{n,r_o} = V_{n,r_i} = 2\cos(5\theta)$.  Since there are no non-zero even indexed Fourier coefficients, it follows that the EMP matrix of every eigenvalue is identically zero; and every eigenvalue branch has slope zero at the initial shape.  In both cases, observe the robust agreement between the slope at zero of the numerically generated eigenvalue branches and the corresponding tangent lines determined analytically from Corollary \ref{C:2DEMPMatrixAnnulus}.

\begin{figure}[h!]
\centering\includegraphics[width = 5in]{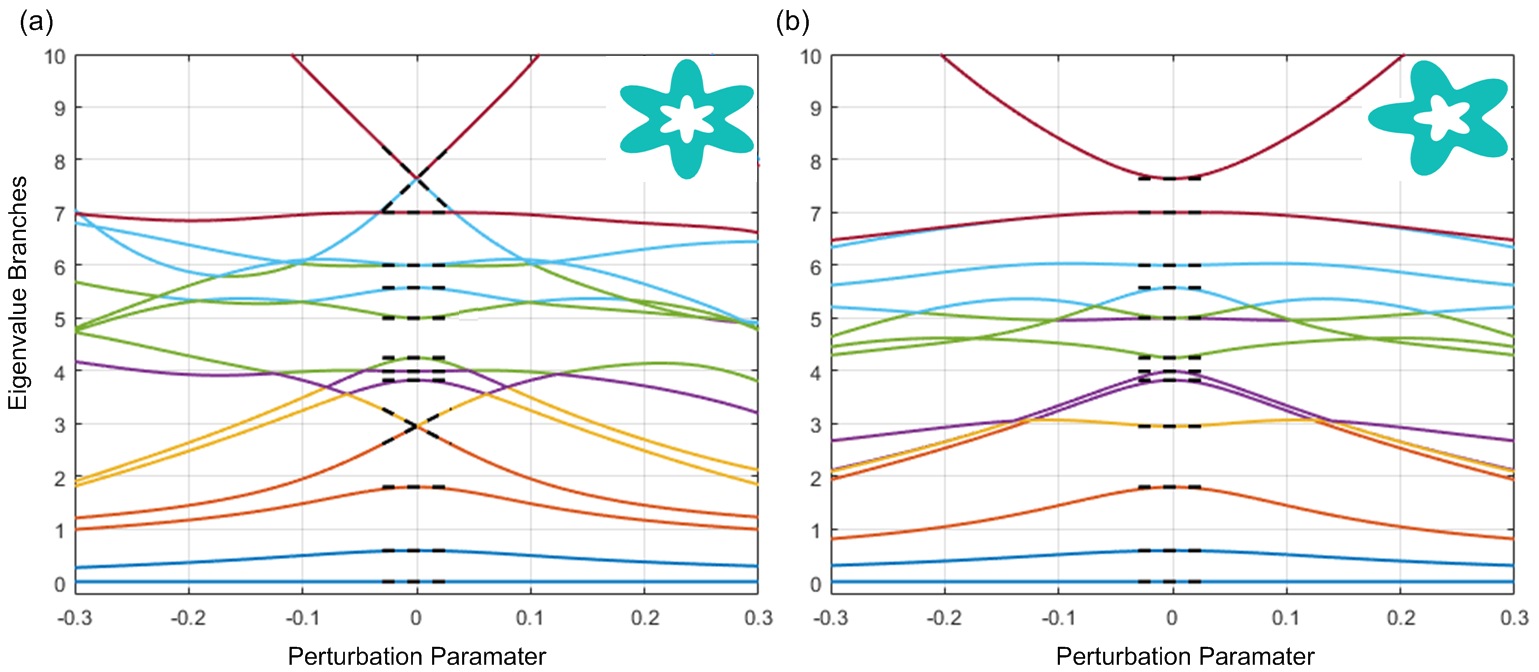}
\caption{ Eigenvalue Branches of the annulus $\mathbb{A}^2_{0.4,1}$ where $L = 7$, $K_o = 28$, and $K_i = 20$.  In  (a)  we use $V_{n,r_i} = V_{n,r_o} = 2\cos(6\theta)$ and in (b) we use  $V_{n,r_i} = V_{n,r_o} = 2\cos(5\theta)$.  }
\label{fig:2DBowtieResponse}
\end{figure}

The local maximization results obtained in Theorems   \ref{T:LocalOptimizationBall} and  \ref{T:LocalOptimizationAnnulus} required that the EMP matrix under consideration be non-zero, raising the question of the possibility of weakening or eliminating this assumption.  In Figure \ref{fig:2DMaxMinPossible}, we see that for unit disk $\mathbb{B}^2$ perturbed by $V$ with normal component $V_{n,1} = \cos(7\theta)$ we have that $(B^2,V)$ appears to locally strictly minimizes both the first and last eigenvalue branches of  $\sigma = 5$, while when $V$ has normal component $V_{n,1} = \sin(5\theta)$ we have that the opposite appears to be true, and $(B^2,V)$ appears to locally strictly maximizes the first and last eigenvalue branches of $\sigma = 5$.  In the same figure we also see that for the annulus $\mathbb{A}^2_{0.4,1}$ perturbed by $V$ with normal components $V_{n,0.4} = \cos(7\theta), V_{n,1} = \cos(7\theta)$ it appears that $(\mathbb{A}^2_{0.4,1},V)$ locally strictly minimizes the first and last eigenvalue branches of $\mu_{5,2}$, while when $V$ has normal component $V_{n,0.4} = 0, V_{n,1} = \sin(5\theta)$ it appears that $(\mathbb{A}^2_{0.4,1},V)$ locally strictly maximizes the first and last eigenvalue branches of $\mu_{5,2}$. Notice that in all four examples the associated EMP matrix is identically zero; and so, based on the numerics, we conclude that our assumption is required to conclude local strict maximization of the bottom eigenvalue branch and local strict minimization of the top eigenvalue branch.

\begin{figure}[h!]
\centering\includegraphics[width = 5in]{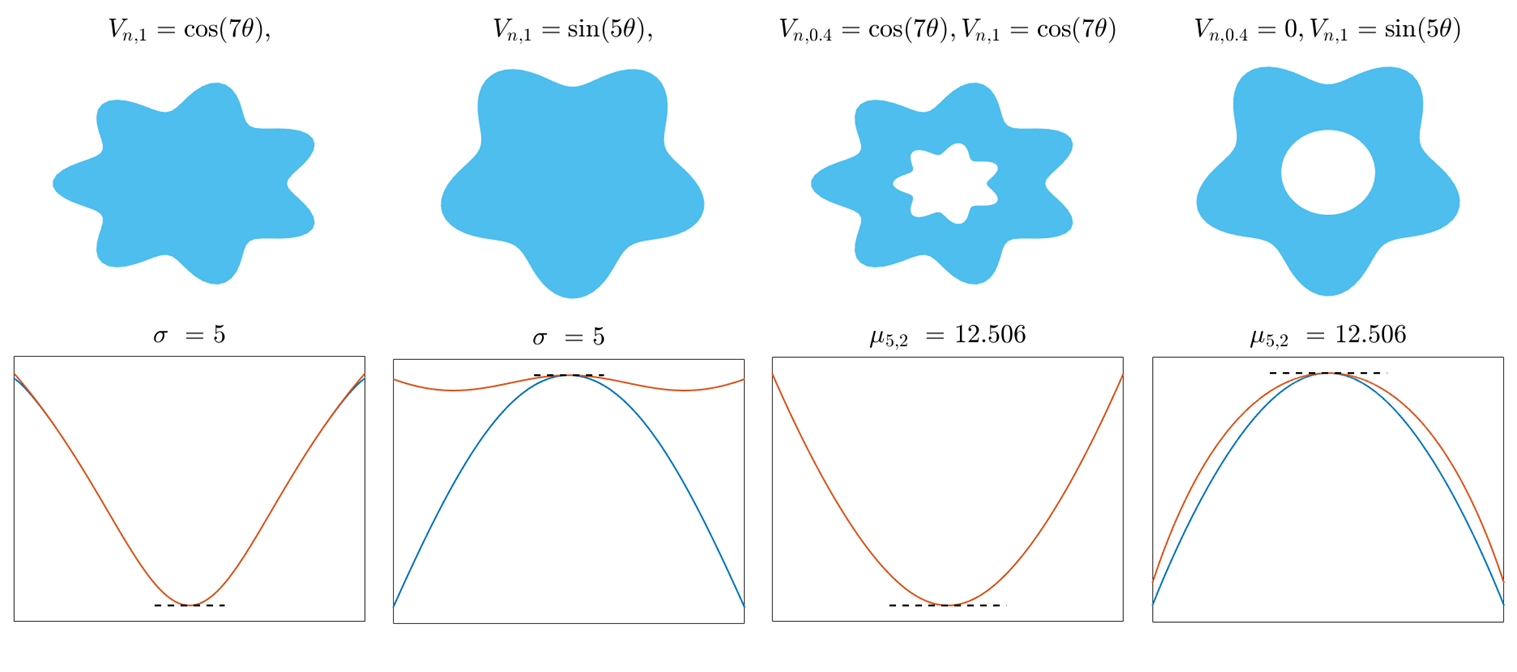}
\caption{Eigenvalue Branches for the eigenvalue $\sigma = 5$ of the unit disk $B^2$ and for the eigenvalue $\mu_{5,2}$ of  $\mathbb{A}^2_{0.4,1}$.  Here we have used $L = 7$, $K_o = 28$, and $K_i = 20$.} 
\label{fig:2DMaxMinPossible}
\end{figure}

\end{example}

\begin{example}\label{E:Example3DNumerics}  In three dimensions we have that the multiplicity of any Steklov eigenvalue $\sigma = \frac{n}{r_o}$ for a disk or $\mu_{n,k}$ for an annulus is given by  $N_{n,2} = 2n+1$, and so the size of the corresponding EMP matrix grows with $n$.  We determine the eigenvalues of these matrices numerically.  We compute the entries of the EMP matrix using the standard orthonormal basis given by Appendix \ref{3DStdBasis}.
\noindent With this choice of basis, we have that  the triple product integrals appearing in  Theorems \ref{T:SphericalTripleProduct}   and \ref{T:AnnularTripleProduct} may be expressed in terms of Wigner-3j symbols as follows
\begin{equation}\label{Wigner3j}
\begin{Small}\text{$
\int_{\mathbb{S}^{2}}Y_{l_1}^{m_1} \overline{Y_{l_2}^{m_2}} Y_{l_3}^{m_3} dS =  
  (-1)^{m_2}\sqrt{\frac{(2l_1 + 1)(2l_2+1)(2l_3 +1)}{4\pi}}
\left(
\begin{matrix}
l_1 & l_2 & l_3 \\
0 & 0 & 0 \\
\end{matrix}
\right)
\left(
\begin{matrix}
l_1 & l_2 & l_3 \\
m_1 & -m_2 & m_3 \\
\end{matrix}
\right).
$}\end{Small} 
\end{equation} 

\noindent A convenient expression for the numerical determination of the eigenvalues of the EMP matrix in the three-dimensional case \cite{Kobi2008}.

An examination of the formulas in Theorems \ref{T:SphericalTripleProduct}   and \ref{T:AnnularTripleProduct}  yields a necessary condition for the non-vanishing of the EMP matrix of a Steklov eigenvalue $\sigma = \frac{n}{r_o}$ or $\mu_{n,k}$ of either a spherical or annular domain, respectively.   Under the assumption that $V$ is boundary component volume preserving at first order, for either a spherical or annular domain a  necessary condition that  the EMP matrix of a Steklov eigenvalue be non-zero is provided by the non-vanishing of at least one Fourier coefficients  $\alpha_{l,m,r_o}$ for spherical domains, and $ \alpha_{l,m,r_i}, \alpha_{l,m,r_o}$ for annular domains, where $2\leq l \leq 2n, l  \ \text{even}$ and $-l\leq m \leq l $, (see \cite[Corollary 2.3]{viator2022steklov} where this condition is also observed).  As in dimension 2, this condition is also sufficient for spherical domains, while not being sufficient in the case of annular domains because of the possibility of cancellation between the inner and outer components of the EMP matrix.

In Figure \ref{fig:3DBowtieResponse} we   visualize the eigenvalue branches and their tangent lines at zero for a selection of eigenvalues $\mu_{n,1}$ of $\mathbb{A}^3_{0.4,1}$.  Where the annular domain  is perturbed by a deformation field $V$ with $V_{n,r_o} = V_{n,r_i} = Y_{8,1}$.  Note that $Y_{8,1}$ denotes the real spherical harmonic of degree 8 and order 1.  In this case the EMP matrix of $\mu_{n,1}$ is identically zero for $n<3$, and the eigenvalue branches all have slope zero at the initial shape.  On the other hand, for $n\geq 4$ the EMP matrix is non-zero, and the eigenvalue branches demonstrate the "bowtie" response to perturbation of the initial shape, characteristic of a non-zero EMP matrix.  Again we observe robust agreement between the analytically determined dotted black tangent lines and the numerically generated eigenvalue  branches.

\begin{figure}[h!]
\centering\includegraphics[width = 5in]{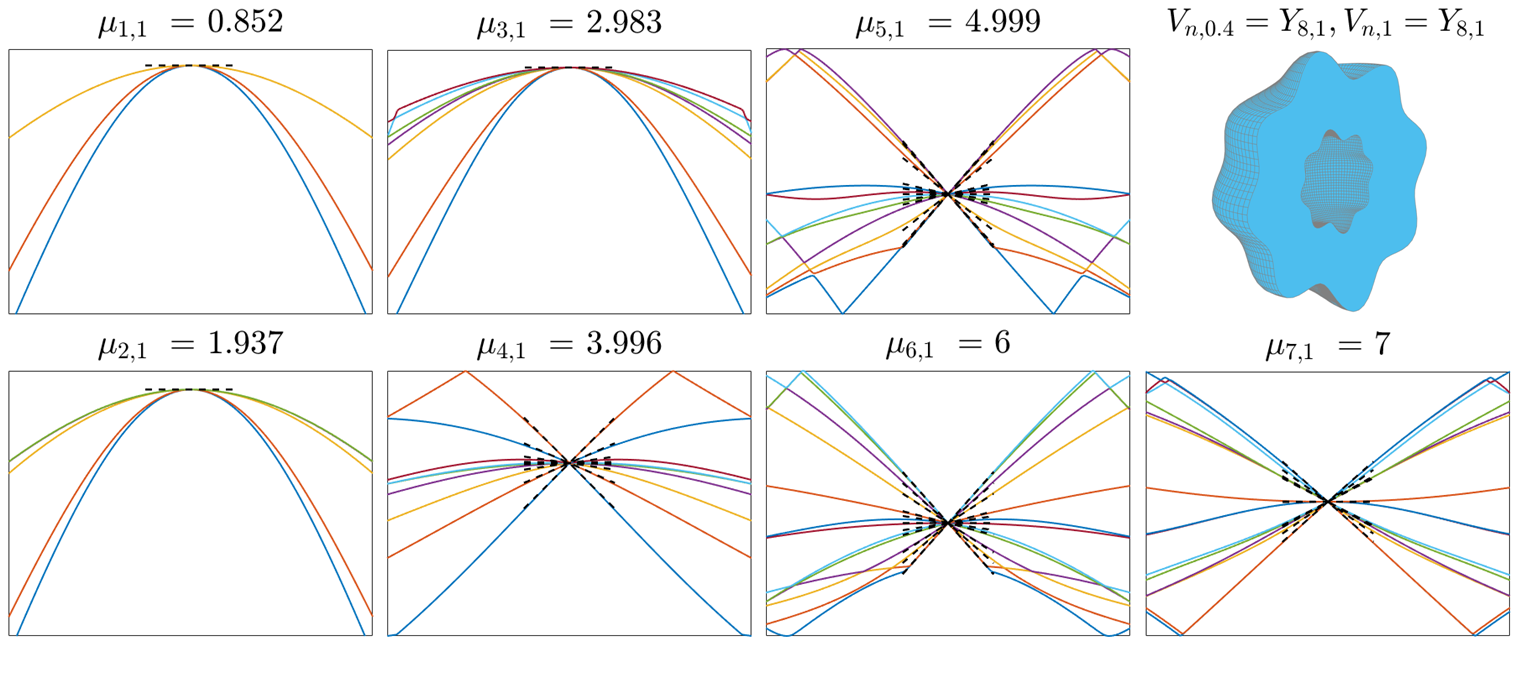}
\caption{$\mathbb{A}^3_{0.4,1}$ Eigenvalue Branches with $V_{n,r_i} = Y_{8,1}, \ V_{n,r_o}  = Y_{8,1}$.   Here we have used $L = 7$, $K_o = 28$, and $K_i = 20$.}  
\label{fig:3DBowtieResponse}
\end{figure}

Finally, in Figure \ref{fig:3DMaxMinPossible}, we see that when its EMP matrix   under perturbation is identically zero, the eigenvalue $\mu_{2,2}$  of  $\mathbb{A}^3_{0.4,1}$ can both locally maximize and  locally minimize both the first and last branch of its eigenvalue branches. Again, the numerics indicate that for local strict optimization the assumption that the EMP matrix be non-zero is required in  \ref{T:LocalOptimizationAnnulus}.  Interestingly an example of similar behavior for the first branch for an eigenvalue of a spherical domain in dimension 3 was not forthcoming, and so the numerics does not immediately support the requirement that the EMP matrix be non-zero.  For instance, we see that perturbing the unit ball by $Y_7^1$ or by $Y_5^5$ results in the unit ball $\mathbb{B}^3$ locally maximizing the first branch of its eigenvalue branches.  

\begin{figure}[h!]
\centering\includegraphics[width = 5in]{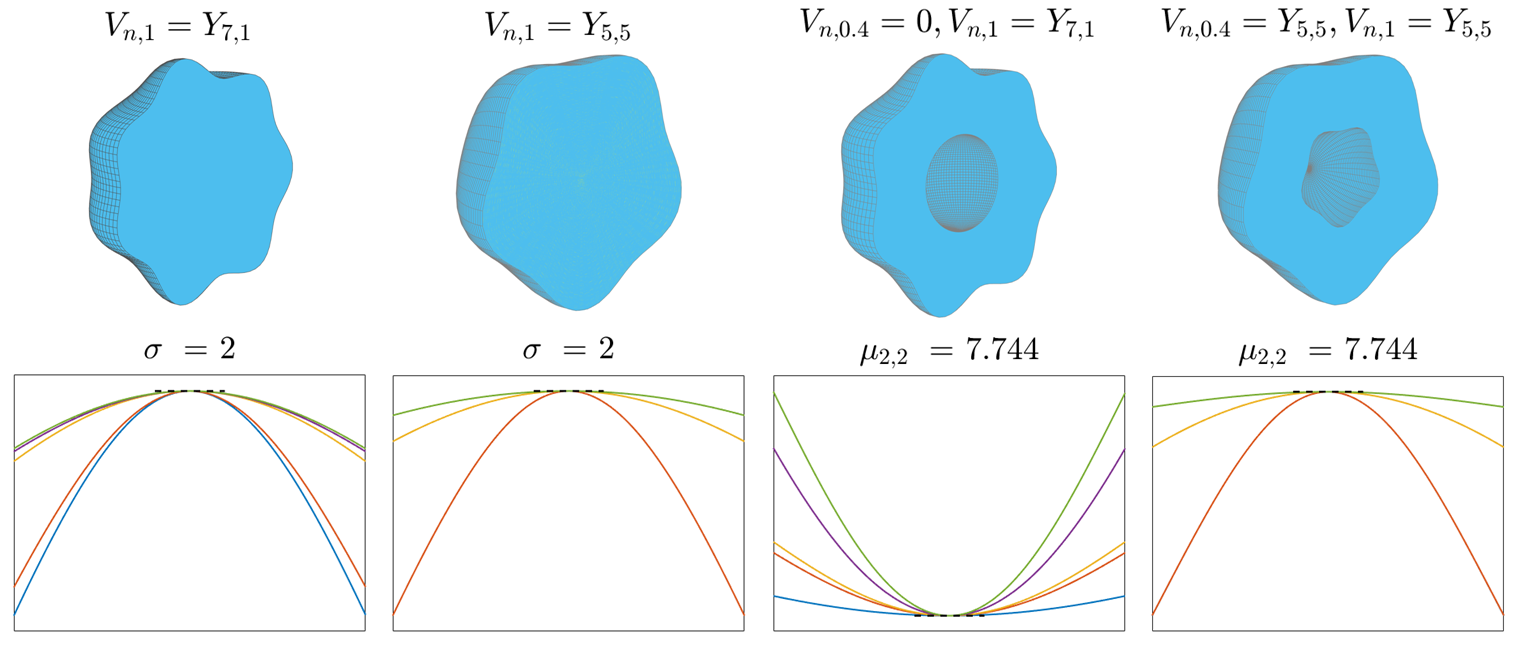}
\caption{Eigenvalue Branches for the eigenvalue $\sigma = 2$ of  of $B^3_1$ and for the eigenvalue $\mu_{2,2}$ of  $\mathbb{A}^3_{0.4,1}$.  Here we have used $L = 7$, $K_o = 28$, and $K_i = 20$.} 
\label{fig:3DMaxMinPossible}
\end{figure}

\end{example}

\section{Conclusion and Future Work}

In this paper, we studied how Steklov eigenvalues vary when a spherical domain or an annular domain in dimensions $d\ge 2$ is perturbed by a sufficiently smooth deformation field. By using a Green-Beltrami identity and that spherical harmonic functions are eigenfunctions of the surface Laplacian, we demonstrated that  the derivatives of multiple Steklov eigenvalue branches are eigenvalues of a matrix whose entries are determined by finite sums of terms that involve the integral of the product of three spherical harmonic functions. It would be of interest to determine if a similar analytic result could be obtained for other symmetric star shaped domains and  their corresponding concentric "annular" versions, for instance ellipsoidal domains and ellipsoidal annular domains.  It would also be of interest to consider even more general domains, and to determine the EMP matrix eigenvalues numerically based directly on Theorem \ref{T:DKLE1}.  

Also, by determining sufficient conditions that imply the trace of its EMP matrix is zero, we show that for a Steklov eigenvalue $\sigma$ of a spherical or annular domain, the pair $(\mathbb{B}^d_{r_o}, V)$, respectively $(\mathbb{A}^d_{r_i,r_o}, V)$, is critical for the eigenvalue  provided the deformation field $V$ is sufficiently smooth.  In addition, we show that if the EMP matrix is not identically zero, then  $(\mathbb{B}^d_{r_o}, V)$, respectively $(\mathbb{A}^d_{r_i,r_o}, V)$, locally maximizes the first branch and locally minimizes the last branch of the eigenvalues branches of $\sigma$.  For spherical domains our sufficient condition, $V$ is sufficiently smooth and volume preserving at first order, is equivalent to the trace of the EMP equaling zero.  In the case of annular domains, the corresponding condition,  $V$ is sufficiently smooth and boundary component volume preserving at first order, implies but is not equivalent to the EMP matrix trace equaling zero; because of the possibility of cancellation between the perturbations on the inner and outer domains.  It would be of interest, in the case of annular domains, to develop a natural geometric condition on $V$ which is equivalent to the vanishing of the EMP matrix trace.  

Finally, numerically, we observe robust agreement between the tangent lines obtained from EMP matrix eigenvalues and the simulated eigenvalue branches obtained using the method of particular solutions.  We also used numerics to investigate the assumption that the EMP matrix not be identically zero, required in our proof that the initial shape locally maximizes the first branch and locally minimizes the last branch of the eigenvalues branches of a given Steklov eigenvalue. Numerically it appears that in two dimensions when the EMP matrix is zero, depending on the deformation field $V$,  both a disk and annulus can either maximize or minimize the first branch among a particular Steklov eigenvalues branches.  So numerically it appears the non-zero EMP matrix assumption is required.  In three dimensions, we observe similar behavior for annular domains; but  we could not find an example where a given eigenvalues first branch is both maximized and minimized by a spherical domain.  Instead, the numerics supports the implication that the first branch is always maximized by a spherical domain. It would be of interest to either establish or refute this result analytically. 

More generally it would be of interest to develop analytic methods which determine the solution to the local optimization problem, even when the EMP matrix is zero; and so a second order analysis like that carried out in Dambrine et al. \cite{dambrine2014extremal} would be of interest.  The results in this paper were achieved making use of elementary properties of the spherical harmonics, in particular no theoretical use of the cumbersome Wigner-3j formulas was needed.  It is our expectation that the techniques developed in this paper will render the second order analysis far more tractable compared to the difficulties encountered when employing the Wigner-3j symbol.

\appendix

\section{Eigenvalue/Eigenfunction Formulas for Annular Domains}\label{AnnEigFormulas}


We discuss the eigenvalues and eigenfunctions for the Steklov problem   on $\mathbb{A}^d_{r_i,r_o}$, the $d$-dimensional annulus with outer radius $r_o$ and inner radius $r_i$.  For annular domains the Steklov eigenvalues are no longer conveniently ordered  according to multiplicity. In general, to each space of spherical harmonics $\mathbb{Y}_l^d$  is associated a pair of eigenvalues $\mu_{l,1}$ and $\mu_{l,2}$ whose eigenspaces are distinguished by a radially dependent multiplicative factor.  For details of the derivation see \cite{martel2014spectre} and for formulas similar to those given below see  \cite[Section 4]{ftouhi2022place}. We breakup our description of the Steklov eigenvalues and corresponding eigenspaces into two cases,  because in dimension $d=2$ we have an eigenfunction with a logarithmic radial term, which does not occur in dimension $d\geq 3$. In what follows  $\{Y_l^m:\ l = 0,1,\cdots;  m = 1,\cdots,N_{l,d}\}$  denotes an arbitrary orthonormal  basis for the vector space of all $d$-dimensional spherical harmonics $\bigoplus_{l=0}^{\infty} \mathbb{Y}_l^d$. \\ 

When $l=0$ we have a pair of eigenvalues given by 
\begin{equation*}
\mu_{0,1} =  0  \quad \text{and} \quad 
\mu_{0,2} =   
\begin{cases}
-\frac{r_i + r_o}{r_i r_o \ln \left(\frac{r_i}{r_o}\right)}, &   \  d =2,\\
\quad & \quad \\
\frac{(d-2)(r_o^{d-1} + r_i^{d-1})}{r_i r_o (r_o^{d-2} - r_i^{d-2})}, &   \ d\geq 3.\\
\end{cases}
\end{equation*}

\noindent   Each eigenvalue $\mu_{0,k}, k =1,2$ is simple with corresponding eigenfunction
\begin{align*}
\begin{Small}\text{$
u_{0,1}^1  (r,\theta_1,\cdots\theta_{d-1})
$}\end{Small}
 &=
\begin{Small}\text{$
\underbrace{ \frac{1}{\sqrt{r_i^{2(d-1)}+r_o^{2(d-1)}}} }_{N(r_i,\mu_{0,1} ,d)}Y_0^1(\theta_1,\cdots,\theta_{d-1})
$}\end{Small}\\
&\quad\\
\begin{Small}\text{$
u_{0,2}^1  (r,\theta_1,\cdots,\theta_{d-1})
$}\end{Small}
 &=
\begin{Small}\text{$
\begin{cases}
\quad \underbrace{ \frac{\text{ln}(r)}{\sqrt{\text{ln}(r_i)^2r_i+ \text{ln}(r_o)^2r_o}} }_{N(r_i,\mu_{0,2} ,2)}Y_0^1(\theta_1) &   \  d =2\\
\quad & \quad\\
\quad \underbrace{\frac{r^{-(d-2)} + 1}{(r_i^{d-1} + r_i) + (r_o^{d-1} + r_o)}}_{N(r_i,\mu_{0,2} ,d)}Y_0^1(\theta_1,\cdots,\theta_{d-1}) &   \  d\geq 3\\
\end{cases}
$}\end{Small}
\end{align*}

\noindent Here and below we write $N(r_i,\mu_{l,k} ,d)$ for the radial dependence $r$ of the dimension $d$ eigenfunctions of $\mu_{n,k}$  orthonormalized on  $\mathbb{S}^1_{r_o} \cup \mathbb{S}^1_{r_i}$.\\

  For $l\geq1$ we have a pair of eigenvalues $\mu_{l,1}$ and $\mu_{l,2}$ given by the zeros of the quadratic equation: 
\begin{equation*}
\begin{Small}\text{$
\mu^2 -B \mu+  \frac{l\left(l+d-2\right)}{r_i r_o} = 0
$}\end{Small}
\end{equation*}
\noindent where
\begin{equation*}
\begin{Small}\text{$
B = \frac{ (l+d-2)\left(r_o^{2l+d-1} +r_i^{2l+d-1}\right)+lr_i r_o\left(r_o^{2l+d-3} +r_i^{2l+d-3}\right)}{r_i r_o\left(r_o^{2l+d-2} - r_i^{2l+d-2}\right)} \\
$}\end{Small}
\end{equation*}

\noindent Each eigenvalue $\mu_{l,k}, k =1,2$ has multiplicity $N_{l,d}$ with corresponding basis of Eigenfunctions 
\begin{equation*}
\begin{Small}\text{$
u_l^m(r,\theta_1,\cdots,\theta_{d-1}) \ =\underbrace{ \left(\frac{a(l,k)}{c(l,k)} r^l +  \frac{b(l,k)}{c(l,k)} r^{-(d+l-2)}\right)}_{N(r_i,\mu_{l,k} ,d)} Y_l^m(\theta_1,\cdots,\theta_{d-1})\ \quad m=1,\cdots,N_{l,d}
$}\end{Small}
\end{equation*}

\noindent where
\begin{align*}
a(l,k) &=  \begin{Small}
\text{$  (d+l-2)\left(r_o^{-(d+l-1)} - r_i^{-(d+l-1)} \right) + \mu_{l,k}\left( r_o^{-(d+l-2)} + r_i^{-(d+l-2)}\right) $}
\end{Small}\\
& \quad\\
b(l,k) &=  \begin{Small}
\text{$  l \left( r_o^{l-1} -  r_i^{l-1} \right) - \mu_{l,k}\left( r_o^{l} + r_i^{l}\right)$}
\end{Small}\\
&\quad\\
c(l,k) &= 
\begin{Small}
\text{$ \sqrt{ r_i^{d-1}\left(a(l,k)r_i^l+ b(l,k)r_i^{-(d+l-2)}\right)^2 + r_o^{d-1}\left(a(l,k)r_o^l+ b(l,k)r_o^{-(d+l-2)}\right)^2}$}
\end{Small}
\end{align*}

\noindent Take note that the particular choice of coefficients $a(l,k)$  and $b(l,k)$ are not unique.  In fact they are obtained by solving a singular system of equations whose determinant set equal to zero gives rise to the quadratic equation for $\mu_{l,k}$.

\section{Standard Orthonormal Basis for Spherical Harmonics}\label{StandardBasisSphHar}


 In this example, for $d\geq 3$, we define a \textit{standard orthonormal basis} for $\bigoplus_{l=0}^{\infty} \mathbb{Y}_l^d$  which has particularly nice behavior under conjugation. In dimension $d=3$, define the standard orthonormal basis for the spherical harmonics
\begin{equation}\label{3DStdBasis}
\begin{Small}\text{$
Y_l^m(\theta,\phi)= \sqrt{\frac{(2l+1)(l-m)!}{4\pi(l+m)!}}P_l^m(\cos\theta)e^{im\phi} \quad \text{for}\  l\geq 0,\  m = -l, \cdots, l.
$}\end{Small} 
\end{equation} 
\noindent Where $P_l^m(x)$ denotes an associated Legendre polynomial (see \cite[Section 3.4]{AveryHypSphHar}.   Here $l$ is the degree of the spherical harmonics and $m$ enumerates the particular basis elements for the  space $\mathbb{Y}_l^d$.  Notice that the index $m$ ranges from $-l$ to $l$.  This choice allows easy representation of the fact that the standard basis has symmetry under conjugation.  Indeed, it is well known that the associated Legendre polynomials satisfy the condition $P_l^m(x) = (-1)^m\frac{(l+m)!}{(l-m)!}P_l^{-m}$ from which it follows that the standard basis elements satisfy the following conjugation relation
\begin{equation}\label{3DStdBasisParity}
\begin{Small}\text{$
\overline{Y}_l^m  \ = \ (-1)^m Y_l^{-m}.
$}\end{Small} 
\end{equation} \\

For dimension $d\geq 4$, we may build up a standard basis recursively starting with the basis elements in \ref{3DStdBasis} (see \cite[Sections 3.5-3.6]{AveryHypSphHar}) and we obtain the following

\begin{equation}\label{dDStdBasis}
\begin{Small}\text{$
Y_{\mu_1,\cdots,\mu_{d-3},l,m}(\theta_1,\cdots,\theta_{d-2},\phi)= N(\vec{\mu},l,m)\left[\prod_{j=1}^{d-3} C_{\mu_j - \mu_{j+1}}^{\alpha_j+\mu_{j+1}}\left(\cos\theta_j\right)\left(\sin \theta_j\right)^{\mu_j +1}\right] Y_l^m(\theta_{d-2},\phi)
$}\end{Small} 
\end{equation}
\noindent Where $\mu_1$ is the degree of the spherical harmonic basis element, the indices $\vec{\mu},l,m = \mu_1,\cdots,\mu_{d-3}, l, m$ satisfy $\mu_1\geq\mu_2\geq \cdots \geq \mu_{d-3}\geq l \geq \vert m \vert$, and $2\alpha_j = d - j - 1$. The coefficient $N(\vec{\mu},l,m)$ normalizes $Y_{\vec{\mu},l,m}$ in $L^2(\mathbb{S}^{d-1})$ and $C_{\mu_j - \mu_{j+1}}^{\alpha_j+\mu_{j+1}}$ are Gegenbauer polynomials. The normalizing coefficient and the product of the Gegenbauer polynomials are real quantities; and so, it follows form (\ref{3DStdBasisParity}) that the higher dimensional standard basis elements also satisfy a conjugation relation
\begin{equation}\label{dDStdBasisParity}
\begin{Small}\text{$
\overline{Y}_{\vec{\mu},l,m} \ = \ (-1)^m Y_{\vec{\mu},l,-m}.
$}\end{Small} 
\end{equation} \\

Making use of (\ref{dDStdBasisParity}) we prove the following conjugation property for real-valued function on the sphere.

\begin{prop} \label{ConjProperty} If $f:\mathbb{S}^{d-1}\to \mathbb{R}$  has Fourier-Laplace expansion in the standard orthonormal basis
\begin{equation*}
\begin{Small}\text{$
f = \sum_{\mu_1 = 0}^{\infty}\ \sum_{{\mu_1,\cdots,\mu_{d-3},l,m}} \alpha_{\vec{\mu},l,m}Y_{\vec{\mu},l,m}
$}\end{Small}
\end{equation*}
then we have the following conjugation formula
\begin{equation}\label{ConjFormula}
\begin{Small}\text{$
\overline{\alpha}_{\vec{\mu},l,m}\overline{Y}_{\vec{\mu},l,m} \ = \ \alpha_{\vec{\mu},l,-m}Y_{\vec{\mu},l,-m}  
$}\end{Small}
\end{equation}
\end{prop}

\begin{proof}
Below we first use that $f$ is real-valued, second, we use \ref{dDStdBasisParity}, and third we reindex the second sum and combine like terms. 

\begin{align*}
\begin{Small}\text{$
0
$}\end{Small}
 \ &= \  
\begin{Small}\text{$
 \sum_{\mu_1 = 0}^{\infty}\ \sum_{{\mu_1,\cdots,\mu_{d-3},l,m}} \alpha_{\vec{\mu},l,m}Y_{\vec{\mu},l,m} -  \sum_{\mu_1 = 0}^{\infty}\ \sum_{{\mu_1,\cdots,\mu_{d-3},l,m}}\overline{\alpha}_{\vec{\mu},l,m}\overline{Y}_{\vec{\mu},l,m}   
$}\end{Small}  \\
&= \ 
\begin{Small}\text{$
 \sum_{\mu_1 = 0}^{\infty}\ \sum_{{\mu_1,\cdots,\mu_{d-3},l,m}} \alpha_{\vec{\mu},l,m}Y_{\vec{\mu},l,m} -  \sum_{\mu_1 = 0}^{\infty}\ \sum_{{\mu_1,\cdots,\mu_{d-3},l,m}}(-1)^m\overline{\alpha}_{\vec{\mu},l,m} Y_{\vec{\mu},l,-m}
$}\end{Small}  \\
&= \ \begin{Small}\text{$
 \sum_{\mu_1 = 0}^{\infty}\ \sum_{{\mu_1,\cdots,\mu_{d-3},l,m}} \left[ \alpha_{\vec{\mu},l,m} - (-1)^m\overline{\alpha}_{\vec{\mu},l,-m}\right]Y_{\vec{\mu},l,m} 
$}\end{Small} 
\end{align*}
\noindent Because 0 has a unique Fourier-Laplace expansion we conclude
\begin{equation*}
\begin{Small}\text{$
\alpha_{\vec{\mu},l,m} - (-1)^m\overline{\alpha}_{\vec{\mu},l,-m} \ = \ 0
$}\end{Small}
\end{equation*}
\noindent So we have $\overline{\alpha}_{\vec{\mu},l,m} =  (-1)^m\alpha_{\vec{\mu},l,-m}$ which together with  $\overline{Y}_{\vec{\mu},l,m} =  (-1)^m Y_{\vec{\mu},l,-m}$ gives the result.
\end{proof}

\subsection*{Acknowledgements} 
We would like to acknowledge helpful discussions with Chee Han Tan and Robert Viator on asymptotic analysis for Steklov eigenvalue problems in general dimensions.

\subsection*{Data Availability Statement} 
The research datasets/codes associated with this article are available in Zenodo, under the reference \\
doi.org/10.5281/zenodo.10034741 \cite{Schroeder_2023_10034741}.

\bibliographystyle{plain}
\bibliography{References_revision}

\end{document}